\documentclass[12pt,reqno]{amsart}
\usepackage{latexsym,amsmath,amssymb, amsthm, mathscinet}
\usepackage{cases, verbatim}

\usepackage{amsfonts,amsmath,amsthm, amssymb}
\usepackage{cases}
\usepackage[usenames]{color}
\usepackage{enumerate}
\usepackage{bm}
\usepackage{graphicx}

\theoremstyle{plain}
\newtheorem{thm}{Theorem}[section]

\newtheorem{defn}[thm]{Definition}

\newtheorem{lem}[thm]{Lemma}
\newtheorem{cor}[thm]{Corollary}
\newtheorem{prop}[thm]{Proposition}
\theoremstyle{remark}
\newtheorem{rem}{\bf{Remark}}
\numberwithin{equation}{section}

\newcommand{\R}{\mathbb{R}}

\newcommand{\ol}{\overline}

\newcommand{\supp}{{\rm supp}\,}

\newcommand{\dist}{{\rm dist}\,}

\begin{document}
\title[Selection of Critical Solutions]{Selection of Critical Solutions in Non-Monotone Vanishing Discount Problems}

\author{Panrui Ni}
\address[P. Ni]{
Department 1: Department of Mathematics, Faculty of Science and Engineering, Waseda University, 3-4-1 Okubo, Shinjuku-ku, Tokyo, 169-8555, Japan; Department 2: Shanghai Center for Mathematical Sciences, Fudan University, Shanghai 200438, China}
\email{panruini@gmail.com}

\author{Jun Yan}
\address[J. Yan]{
	School of Mathematical Sciences, Fudan University, Shanghai 200433, China}
\email{yanjun@fudan.edu.cn}

\author{Maxime Zavidovique}
\address[M. Zavidovique]{
	Sorbonne Universit\'e, Universit\'e de Paris Cit\'e, CNRS, \ Institut de Math\'ematiques de Jussieu-Paris Rive Gauche, Paris 75005, France}
\email{mzavidovi@imj-prg.fr}


\makeatletter
\@namedef{subjclassname@2020}{\textup{2020} Mathematics Subject Classification}
\makeatother



\begin{abstract}
We study a generalized vanishing discount problem for Hamilton--Jacobi equations without assuming monotonicity either pointwise or in the averaged sense with respect to all Mather measures. Specifically, we consider
\[
\lambda a(x)u(x)+H\big(x,Du(x)\big)=c_0,
\]
and assume that, for every sufficiently small $\lambda>0$, this equation admits a strict subsolution bounded from below uniformly with respect to $\lambda$. We prove that, under a smallness assumption on the Mather quotient, the maximal viscosity solution converges uniformly as $\lambda\to 0^+$ under a sign condition on the discount coefficient $a(x)$ determined by a prescribed part of the Aubry set. The limit is characterized explicitly in terms of the Peierls barrier and the selected Mather measures. As an application, the elementary solution of the critical Hamilton--Jacobi equation
\[
H\big(x,Du(x)\big)=c_0
\]
associated with any isolated static class can be realized as the vanishing discount limit. This provides the first mechanism for selecting multiple critical solutions in vanishing discount problems beyond the classical monotonicity framework, showing that different families of Mather measures yield different limiting critical solutions.

We finally propose a variant of the discounted problems for which the strict subsolution hypothesis is automatically verified and hence applies to more general Hamiltonians $H$ and functions $a$.

\end{abstract}

\date{\today}

\maketitle



\section{Introduction}

In this paper, we investigate whether different solutions of
\begin{equation}\label{E0}\tag{HJ$_0$}
  H\big(x,Du(x)\big)=c_0\quad \textrm{in}\quad M
\end{equation}
can be obtained as the limit of the maximal viscosity solution of
\begin{equation}\label{E}\tag{E$_\lambda$}
  \lambda a(x)u(x)+H\big(x,Du(x)\big)=c_0 \quad \textrm{in}\quad M
\end{equation}
as $\lambda\to 0^+$ when we choose different $a(x)$. Here $M$ is a closed Riemannian manifold, $H\in C(T^*M)$ is convex and superlinear in the gradient variable, $a(x)$ is continuous, and $c_0$ is the critical value associated with $H$. Since the equation is non-monotone in $u$, existence and compactness are not automatic. We therefore assume that, for every sufficiently small $\lambda>0$, \eqref{E} admits a strict subsolution bounded from below uniformly with respect to $\lambda$. Throughout the paper, solutions, subsolutions, and supersolutions are meant in the viscosity sense and assumed continuous. Relevant definitions are recalled in Section \ref{pre}.

To obtain the selection result, we introduce a dynamical concept called \emph{static classes}, a natural partition of the Aubry set into equivalence classes following Ma\~n\'e. The corresponding quotient set is called the \emph{Mather quotient set}. In this paper, we assume the Mather quotient set has 1-dimensional Hausdorff measure zero. Let $\widetilde{\mathfrak{M}}$ be the set of all Mather measures. For $\tilde \mu\in \widetilde{\mathfrak{M}}$, we denote by $\mu$ its projection. Assume there is an isolated subset $\mathcal A_1$ of the projected Aubry set $\mathcal A$. We prove that by setting
\[a(x)>0\text{ on }\mathcal A_1\text{ and }\int_{TM} a(x)\, d\tilde\nu<0\text{ for all }\tilde\nu\text{ with }\supp(\nu)\cap \mathcal A_1=\emptyset,\]
the maximal viscosity solution of \eqref{E} converges uniformly to
\[\inf_{\tilde{\mu}\in\widetilde{\mathfrak M}_1}\frac{\int_{TM} h^\infty(y,x)a(y)\, d\tilde\mu(y,v)}{\int_{TM} a(y)\, d\tilde\mu(y,v)}\quad \text{as }\lambda\to 0^+,\]
where $h^\infty(\cdot,\cdot)$ is defined in \eqref{barr} and $\widetilde{\mathfrak M}_1$ is the set of all Mather measures $\tilde\mu$ with $\supp(\mu)\subset\mathcal A_1$. The limit above may be different from the one selected by the limit in \cite{Da4}. As shown in \cite{Ber}, there is a one-to-one correspondence between static classes and elementary solutions of \eqref{E0}. If we choose $\mathcal A_1$ to be an isolated static class, we can select the elementary solution associated with this static class; see Corollary \ref{cor1} below.

Recently, \cite{QC} studied the selection problem for \eqref{E0} through a two-step perturbation involving a vanishing discount term and a small potential perturbation. This approach was the first to select different solutions of \eqref{E0}, where the potential term plays a key role in removing unwanted minimizing measures. See also \cite{QZ} for the degenerate second-order setting. Related selection phenomena for mechanical Hamilton--Jacobi equations are studied in \cite{MNT}, where the relative scaling between the vanishing discount and viscosity terms leads to different limiting
critical solutions. In the present work, by carefully choosing the signs of the discount coefficient $a(x)$, one can select different solutions relying solely on the vanishing discount term. This result emphasizes a purely dynamical mechanism driving the selection process and offers an intrinsic perspective on the problem.

\subsection*{History of the problem}

The existence of solutions to the stationary equation
\begin{equation}\label{hjc}
H\big(x,Du(x)\big)=c \quad \textrm{in}\quad \mathbb T^d
\end{equation}
is known as the \emph{cell problem}, which is a central issue in the theory of Hamilton--Jacobi equations. 
This problem was solved via the so-called ergodic (or discounted) approximation in \cite{hom}. 
Let $\lambda>0$ and let $u_\lambda$ be the unique solution of
\[
\lambda u(x)+H\big(x,Du(x)\big)=0 \quad \textrm{in}\quad \mathbb T^d.
\]
It was shown in \cite{hom} that there exists a sequence $\lambda_k\to0^+$ such that $-\lambda_k u_{\lambda_k}$ converges uniformly to a constant $c_0$, and $u_{\lambda_k}-\min_{\mathbb T^d}u_{\lambda_k}$ converges uniformly to a solution of \eqref{hjc} with $c=c_0$. 
Moreover, $c_0$ is the unique constant for which \eqref{hjc} admits solutions and is called the \emph{critical value}.

At that time, it was unclear whether different sequences $\lambda\to0^+$ would lead to the same limit. 
This question was first studied under restrictive assumptions in \cite{G,IM}. 
A complete positive answer was later obtained in \cite{Da4}, where the uniform convergence of the unique solution $u_\lambda$ of
\begin{equation}\label{lde}\tag{${\textrm{HJ}}_\lambda$}
\lambda u(x)+H\big(x,Du(x)\big)=c_0 \quad \textrm{in}\quad M
\end{equation}
as $\lambda\to0^+$ was established. 
Here $M$ is a closed connected manifold and $H$ is continuous, coercive and convex in the momentum variable. 
This type of problem is referred to as the \emph{vanishing discount problem}. 
The convexity of $H$ plays a crucial role, and convergence may fail without this assumption; see \cite{Z2}. Quantitative convergence rates for the vanishing discount problem under hyperbolicity assumptions were recently studied in \cite{NiRate}.

The vanishing discount problem has subsequently been studied in a variety of settings. For possibly degenerate second-order equations, an approach based on the nonlinear adjoint method was used in \cite{MT} and later adapted in \cite{Zh}. A different variational framework based on convex duality and viscosity Mather measures was developed in \cite{IMT,IMT2}. Other developments include discrete models \cite{Da1,n1,Su,Zbook}, mean field games \cite{CP,IMWX}, weakly coupled Hamilton--Jacobi systems \cite{Da6,Da5,Ish4,Ish5}, and non-compact manifolds \cite{Da7,Ish6}. 
For nonlinear generalizations, known as the vanishing contact structure problem, we refer to \cite{V3,V1,GMT,V2}. 
Other extensions of the vanishing discount problem have been investigated in \cite{TZ,WZ} and the references therein.

As a degenerate but still monotone case, \cite{Z} studied the convergence of solutions to
\begin{equation}\label{gendis}\tag{$\overline{\textrm{HJ}}_\lambda$}
\lambda a(x)u(x)+H\big(x,Du(x)\big)=c_0 \quad \textrm{in}\quad M,
\end{equation}
where $a(x)\geqslant 0$ on $M$ and $a(x)>0$ on the projected Aubry set of $H$. 
This equation is closely related to the present work and is also connected to optimization problems in economics; see \cite{factor}. Here we also point out a model in \cite{BLL} derived from economics. Although this model is a viscous Hamilton--Jacobi equation, the discount coefficient $a(x)$ is assumed to change sign, which is quite related to our case in the present paper. 
Later, the degenerate vanishing contact structure problem was investigated in \cite{CFZZ}, where it was shown that Mather measures play a central role in the convergence analysis.

It is worth pointing out that apart from \cite{BLL}, all the works mentioned above rely on a non-decreasing (monotonicity) assumption with respect to the unknown function $u$. 
Once this assumption is violated, many fundamental tools break down: solutions may fail to exist, comparison principles no longer hold, and uniqueness is generally lost. 
In \cite{Da3,V4}, the convergence of minimal solutions of \eqref{lde} as $\lambda\to0^-$ was studied under restrictive hypotheses. 
For genuinely non-monotone vanishing discount problems, the first example exhibiting both convergent and divergent families of solutions was provided in \cite{n3}, revealing phenomena absent from the monotone theory. 
Subsequently, the generalized vanishing discount problem without the non-decreasing assumption was further developed in \cite{DNYZ}.

\subsection*{Motivation of this paper}

Previous works mainly focus on the convergence of discounted solutions\footnote{A \emph{discounted solution} refers to a viscosity solution of the generalized discounted equation \eqref{gendis}, while a \emph{critical solution} refers to a viscosity solution of \eqref{E0}.}. 
In contrast, the present paper addresses a more refined problem: how to select \emph{different critical solutions} via limits of discounted solutions. 
As observed in \cite{DNYZ}, there exists a certain critical solution that is the only possible limit for all bounded discounted solutions, so previous works could only select this particular solution. 

In \cite{GL}, a selection principle based on the Freidlin--Wentzell large deviation principle was proposed, yielding a critical solution different from the one obtained via the vanishing discount process. 
For instance, as in \cite[Section~5.3]{GL}, we consider
\[
  \lambda u(x)+u'(x)\big(u'(x)-U'(x)\big)=0 \quad \textrm{in}\quad \mathbb S^1\simeq [0,1),
\]
where $U\in C^\infty(\R)$ is skew periodic. The unique discounted solution is $u_\lambda\equiv 0$, so the selected critical solution $u_0$ is trivial. 
In this paper, we show that the \emph{generalized vanishing discount process} can be used to select various critical solutions. Furthermore, we show that under appropriate hypotheses it selects prescribed critical solutions called elementary solutions; see Section \ref{ex1}. 
This represents a natural development of the vanishing discount problem, highlighting for the first time how the dynamical structure of the Aubry set governs the asymptotic behavior of discounted solutions.

According to \cite{DNYZ,n3}, one cannot in general expect all discounted solutions of \eqref{E} to converge as $\lambda$ vanishes. 
This motivates us to consider the convergence of the \emph{maximal discounted solution} of \eqref{E}. Here we note that solutions may fail to exist or may lose compactness in the non-monotone setting. Our hypothesis is formulated differently from the general convergence result
in \cite{DNYZ}, where an equi-bounded family of discounted solutions is
assumed a priori. Here we instead assume that, for every sufficiently small $\lambda>0$,
\eqref{E} admits a strict subsolution $v_\lambda$ uniformly bounded from below. 
Sections \ref{ex1} and \ref{ex2} show that hypothesis {\rm (S)} can hold for \eqref{E}. Section
\ref{sec:4.3} shows that, for a generalized problem \eqref{EA}, the analogue of
{\rm (S)} may fail for $A=0$, which corresponds to \eqref{E}, and even for small positive $A$.
On the other hand, Remark \ref{rem:shift} shows that it holds for
all sufficiently large $A$.
A second key novelty of this paper is that we establish convergence when the integral of the discount coefficient $a(x)$ may be negative for some Mather measures. 
This introduces essential technical challenges, as previous works \cite{CFZZ,DNYZ} rely on the assumption that $a(x)$ is strictly positive when integrated against all Mather measures. Removing this monotonicity condition imposed on all Mather measures gives rise to genuinely new selection phenomena.

Finally, a key ingredient in our analysis is a large-time behavior result (Lemma \ref{stab}), building on earlier work \cite{n2}. To our knowledge, this reveals for the first time a connection between the vanishing discount problem and the large-time behavior.

\subsection*{Statement of the main result}

Let $M$ be a closed connected smooth manifold. Denote by $TM$ and $T^*M$ the tangent and cotangent bundles of $M$, with points $(x,v)\in TM$ and $(x,p)\in T^*M$. Let $\|\cdot\|_x$ and $d(\cdot,\cdot)$ be the norm on $T_xM$ and the distance function induced by the Riemannian metric respectively, and let $\pi:TM\to M$ be the canonical projection. Assume that $H:T^*M\to\mathbb R$ is continuous and satisfies:
\begin{itemize}
\item [(H1)] {\bf Convexity:} $p\mapsto H(x,p)$ is convex for all $(x,p)\in T^*M$.
\item [(H2)] {\bf Superlinearity:} $\lim\limits_{\|p\|_x\to+\infty} H(x,p)/\|p\|_x = +\infty$ uniformly in $x$.
\end{itemize}

The associated Lagrangian $L:TM\to\mathbb R$ is defined by
\[
L(x,v)=\sup_{p\in T^*_xM} \big(\langle p,v\rangle_x - H(x,p)\big).
\]
Under (H1), (H2), $L(x,v)$ is continuous on $TM$, convex in $v$ for all $(x,v)\in TM$, and $\lim\limits_{\|v\|_x\to+\infty} L(x,v)/\|v\|_x = +\infty$ uniformly in $x$;
see \cite[Appendix A.2]{CS}.

Let $h^\infty : M \times M \to \mathbb{R}$ be the Peierls barrier defined in \eqref{barr}. 
The following pseudo-distance was first introduced in \cite{Ma3}; see also \cite{global} for a detailed discussion:
\[
d_H(x,y) := h^\infty(x,y) + h^\infty(y,x).
\]
Let $\mathcal A$ be the projected Aubry set defined in \eqref{pA}. 
We define an equivalence relation on $\mathcal A$ by declaring $x \sim y$ if $d_H(x,y)=0$. 
The equivalence classes are called the \emph{static classes}. So $d_H$ is a distance function on the quotient set formed by static classes. This quotient, endowed with the distance $d_H$, is denoted by $\mathcal Q$ and is called the Mather quotient.
Assume:

\begin{itemize}
\item [($\diamond$)] The Mather quotient set $\mathcal Q$ has 1-dimensional Hausdorff measure zero.
\end{itemize}
The assumption $(\diamond)$ holds for instance if we have one of the following assumptions:
\begin{itemize}
\item [(1)] the Mather quotient is countable;
\item [(2)] the dimension of the manifold $M$ is 2 and the Hamiltonian is $C^2$ Tonelli;
\item [(3)] the dimension of the manifold $M$ is 3 and the Hamiltonian $H(x,p)$ is $C^4$ Tonelli.
\end{itemize}
The second and third cases follow from the work in \cite{FFR}, where the authors prove that in cases (2) and (3) the 1-dimensional Hausdorff measure of the Mather quotient set is zero. A special case of assumption $(\diamond)$ is provided by some symmetric Hamiltonians satisfying
$H(x,p)=H(x,-p)$. 
In this case, 
\[
\mathcal A=\Big\{x\in M\mid H(x,0)=\max_{y\in M} H(y,0)\Big\},
\]
see \cite[Proposition 7]{Da3}. 
Consequently, if $H(x,0)$ admits only finitely or countably many maximizers, then the Aubry set consists of finitely or countably many points, each forming a distinct static class. Therefore, assumption $(\diamond)$ is satisfied in this setting.

Assumption $(\diamond)$ is closely related to the total disconnectedness
\footnote{That is, each connected component of the set consists of a single point.}
of the Mather quotient set induced by the equivalence relation defined above; see, for instance, \cite{FFR,Ma1,So}. 
In general, however, this property does not hold; cf. \cite{BIK,Ma4}.

\medskip

The main result of this paper is as follows.

\begin{thm}\label{thm1}
Assume \emph{(H1)}, \emph{(H2)}, and \emph{($\diamond$)}. Assume $\mathcal A$ is not connected and let us consider a nonempty subset $\mathcal A_1\subsetneq \mathcal A$ such that
\[\dist(\mathcal A_1,\mathcal A\setminus \mathcal A_1)\geqslant \delta>0,\]
where $\dist(\cdot,\cdot)$ is the distance function between two subsets of $M$ induced by $d$. Define
\[\widetilde{\mathfrak M}_1:=\{\tilde\mu\in\widetilde{\mathfrak M}\mid \supp(\mu)\subset \mathcal A_1\}.\]
Let $a\in C(M)$ satisfy
\begin{itemize}
\item[(A)] $a(x)>0$ on $\mathcal A_1$ and
\[\int_{TM} a(x)\, d\tilde\nu<0\text{ for all }\tilde\nu\in\widetilde{\mathfrak M}\text{ with }\supp(\nu)\cap \mathcal A_1=\emptyset.\]
\item[(S)] There exist $\lambda_0>0$ and $C>0$ such that, for every $\lambda\in(0,\lambda_0)$, equation \eqref{E} admits a strict viscosity subsolution $v_\lambda$ satisfying $v_\lambda\geqslant -C$ on $M$.
\end{itemize}
Then, for every $\lambda\in(0,\lambda_0)$, \eqref{E} admits a maximal solution $u_\lambda$, and $u_\lambda$ converges uniformly as $\lambda\to0^+$ to
\[u_0:=\sup_{w\in\mathcal S}w,\]
where
\[\mathcal S:=\bigg\{w\mid w\text{ is a subsolution of }\eqref{E0}\text{ with }\int_{TM}a(x)w(x)\,d\tilde\mu\leqslant0\quad\text{for all }\tilde\mu\in\widetilde{\mathfrak M}_1\bigg\}.\]
Moreover,
\[u_0(x)=\inf_{\tilde{\mu}\in\widetilde{\mathfrak M}_1}\frac{\int_{TM} h^\infty(y,x)a(y)\, d\tilde\mu(y,v)}{\int_{TM} a(y)\, d\tilde\mu(y,v)}.\]
\end{thm}
The next important remark explains a way to modify the problem in order to force hypothesis {\rm (S)} to hold for general Hamiltonians and functions $a$. 
\begin{rem}\label{rem:shift}
As shown in Sections \ref{ex1} and \ref{ex2} below, hypothesis {\rm (S)}
may hold for \eqref{E} but Section \ref{sec:4.3} shows a rather simple example for which it does not. We also record a useful variant of Theorem \ref{thm1}. For $A>0$, consider
the shifted equation
\begin{equation}\label{EA}\tag{$E_{\lambda,A}$}
  \lambda a(x)u(x)+H\bigl(x,Du(x)\bigr)-A\lambda=c_0
  \quad\text{in }M.
\end{equation}
Let $q$ be any viscosity solution of \eqref{E0}. If
\[
A>\max_{x\in M}a(x)q(x),
\]
then $q$ is a strict subsolution of \eqref{EA} for every $\lambda>0$, since
\[
\lambda a(x)q(x)+H(x,Dq)-A\lambda
\leqslant
c_0-\lambda\left(A-\max_M aq\right)
<c_0.
\]
In particular,
\[
A>\|a\|_\infty\|q\|_\infty
\]
is a sufficient condition.

The proof of Theorem \ref{thm1} applies verbatim to \eqref{EA} for $A$ fixed. In this
case, the constraint in the definition of $\mathcal S$ becomes
\[
\int_{TM}a(x)w(x)\,d\tilde\mu\leqslant A,
\]
and the representation formula becomes
\[
u_0^A(x)
=
\inf_{\tilde\mu\in\widetilde{\mathfrak M}_1}
\frac{
\int_{TM}h^\infty(y,x)a(y)\,d\tilde\mu(y,v)+A
}{
\int_{TM}a(y)\,d\tilde\mu(y,v)
}.
\]
Example~3 below shows that, although sufficiently large $A$ always yields
the required strict-subsolution condition for \eqref{EA}, this condition
may fail when $A=0$ or when $A>0$ is too small.
\end{rem}

Notice that such a function $a$ satisfying hypothesis (A) always exists. For instance, one may take
\[
a(x):=\dist(x,\mathcal A\setminus\mathcal A_1)-\dist(x,\mathcal A_1).
\]
Indeed,
\[
a\geqslant\delta\quad\text{on }\mathcal A_1,
\qquad
a\leqslant-\delta\quad\text{on }\mathcal A\setminus\mathcal A_1.
\]
Since the projected Mather set is contained in $\mathcal A$,
condition {\rm (A)} follows.

Here we also note that $\widetilde{\mathfrak M}_1\neq \emptyset$. In fact, by \cite[Section 4]{gen}, for $x\in\mathcal A_1$, there is a curve $\gamma:\mathbb R\to M$ with $\gamma(0)=x$ and $\gamma(\mathbb R)\subset\mathcal A$ such that for every solution $v_0$ of \eqref{E0}, $\gamma$ is $v_0$-calibrated. Since $\gamma$ is continuous, $\gamma(\mathbb R)\subset\mathcal A_1$. We define the probability measure $\tilde\mu_\gamma$ distributed on $(\gamma,\dot\gamma)$ similarly to \eqref{muz}, and the corresponding weak limit $\tilde\mu_*$. Similar to Lemma \ref{mu*M}, one can show that $\tilde\mu_*\in\widetilde{\mathfrak M}$, which is supported on $\mathcal A_1$ since $\gamma(\mathbb R)\subset\mathcal A_1$.

\medskip
Let $\mathcal C\subsetneq \mathcal A$ be a static class. The elementary solution of \eqref{E0} associated with $\mathcal C$ is, up to a constant, $h^\infty(x_0,x)$ for any $x_0\in \mathcal C$; see Corollary \ref{h-h=c} below. When $\mathcal C$ is isolated in $\mathcal A$, we have
\begin{cor}\label{cor1}
Assume \emph{(H1)}, \emph{(H2)}, and \emph{($\diamond$)}. Let $\mathcal C\subsetneq\mathcal A$ be a static class satisfying
\[\dist(\mathcal C,\mathcal A\setminus\mathcal C)\geqslant\delta>0.\]
Define
\[\widetilde{\mathfrak M}_{\mathcal C}:=\{\tilde\mu\in\widetilde{\mathfrak M}\mid \supp(\mu)\subset\mathcal C\}.\]
Let $a\in C(M)$ satisfy $a(x)>0$ on $\mathcal C$ and
\[\int_{TM} a(x)\, d\tilde\nu<0\text{ for all }\tilde\nu\in\widetilde{\mathfrak M}\text{ with }\supp(\nu)\cap \mathcal C=\emptyset.\]
Assume moreover that hypothesis {\rm (S)} of Theorem \ref{thm1} holds. Then the maximal solution $u_\lambda$ of \eqref{E} converges uniformly as $\lambda\to0^+$ to
\[h^\infty(x_0,x)+C_{\mathcal C},\]
where $x_0\in\mathcal C$ is arbitrary and
\[C_{\mathcal C}=\inf_{\tilde\mu\in \widetilde{\mathfrak M}_{\mathcal C}}
\frac{\int_{TM} a(y)h^\infty(y,x_0)\,d\tilde\mu(y,v)}
{\int_{TM} a(y)\,d\tilde\mu(y,v)}.\]
\end{cor}

For the shifted equation \eqref{EA}, Remark \ref{rem:shift} shows
that the analogue of {\rm (S)} holds for all sufficiently
large $A$. Hence, up to an additive constant, the elementary solution associated with any isolated static class can be realized as a vanishing discount limit.

The aforementioned results are all direct consequences of a more general result that holds without hypothesis ($\diamond$) (only disconnectedness of $\mathcal A$ is needed). This result is obtained following closely an unpublished strategy of proof discovered by Fathi and Iturriaga \cite{FaIt}.

Along the way we also prove a general result about extremal Mather measures that is of independent interest; see Lemma \ref{w1w2}. The proof of this result was obtained with the help of ChatGPT.

\section{Preliminaries}\label{pre}

\subsection*{Viscosity solutions and weak KAM solutions}

In this subsection, we collect several properties of viscosity solutions and weak KAM solutions. We refer the reader to \cite{Barles,guide,Fat12} for further details.

\begin{defn}
Let $G:T^*M\times\mathbb R\to\mathbb R$ be a continuous function and $c\in\mathbb R$. A function $u\in C(M)$ is called a viscosity subsolution (resp. supersolution) of
\begin{equation}\label{hjj}
  G\big(x,Du(x),u(x)\big)=c\quad \textrm{in}\quad M
\end{equation}
if for each $\phi\in C^1(M)$, when $u-\phi$ attains its local maximum (resp. minimum) at $x$, then
\begin{equation*}
  G\big(x,D\phi(x),u(x)\big)\leqslant c,\quad (\textrm{resp}.\ \geqslant c).
\end{equation*}
A continuous function $u$ is called a viscosity solution of \eqref{hjj} if it is both a viscosity subsolution and a viscosity supersolution. If there is a constant $\bar c<c$ such that \[G\big(x,Du(x),u(x)\big)\leqslant \bar c\quad \textrm{in}\quad M\] holds in the viscosity sense, we call $u$ a strict subsolution of \eqref{hjj}.
\end{defn}

\begin{prop}\label{stability}
Let $(G_n)_n$ and $(u_n)_n$ be two sequences of functions in $C(T^*M\times\mathbb R)$ and $C(M)$ respectively. For each $n\in\mathbb N$, $u_n$ is a solution (resp. subsolution, supersolution) of \eqref{hjj} with $G=G_n$. If $u_n\to u$ uniformly and $G_n\to G$ locally uniformly as $n\to+\infty$, then $u$ is a solution (resp. subsolution, supersolution) of \eqref{hjj}.
\end{prop}

\begin{prop}\label{supsub}
If the pointwise supremum $u$ of a family of subsolutions of \eqref{hjj} is finite and continuous, then $u$ is a subsolution of \eqref{hjj}.
\end{prop}

Now we denote by $(x,p,u)$ a point in $T^*M\times\mathbb R$. Assume that $G(x,p,u)$ is continuous, and is convex in $p$ for each $(x,u)\in M\times\mathbb R$.
\begin{prop}\label{ue}
Let $w:M\to\mathbb R$ be a Lipschitz continuous function verifying \[G\big(x,Dw(x),w(x)\big)\leqslant c\] for almost every $x\in M$. Then for every $\varepsilon>0$, there is $w_\varepsilon\in C^\infty(M)$ such that \[\|w-w_\varepsilon\|_\infty\leqslant \varepsilon,\quad G\big(x,Dw_\varepsilon(x),w_\varepsilon(x)\big)\leqslant c+\varepsilon\quad \forall x\in M.\]
\end{prop}

Now we further assume that there is $\Theta>0$ such that
\[|G(x,p,u)-G(x,p,v)|\leqslant \Theta|u-v|,\quad \forall (x,p)\in T^*M,\ \forall u,v\in\mathbb R,\]
and \[\lim_{\|p\|_x\to+\infty}\inf_{x\in M}G(x,p,0)=+\infty,\]
then we can define the associated Lagrangian
\[L_G(x,v,u):=\sup_{p\in T^*_xM}\big(\langle p,v\rangle_x-G(x,p,u)\big).\]

\begin{defn}\label{bws}
A function $u\in C(M)$ is called a backward (resp. forward) weak KAM solution of \eqref{hjj} if
\begin{itemize}
\item [(1)] For each absolutely continuous curve $\gamma:[t',t]\rightarrow M$, we have
\begin{equation*}
  u\big(\gamma(t)\big)-u\big(\gamma(t')\big)\leqslant \int_{t'}^{t}\Big[L_G\Big(\gamma(s),\dot \gamma(s),u\big(\gamma(s)\big)\Big)+c\Big]\, ds.
\end{equation*}
The above condition is denoted by $u\prec L_G+c$.

\item [(2)] For each $x\in M$, there exists an absolutely continuous curve $\gamma_-:(-\infty,0]\rightarrow M$ (resp. $\gamma_+:[0,+\infty)\to M$) with $\gamma_\pm(0)=x$ such that
\begin{align*}
  &u(x)-u\big(\gamma_-(t)\big)=\int_t^0\Big[L_G\Big(\gamma_-(s),\dot \gamma_-(s),u\big(\gamma_-(s)\big)\Big)+c\Big]\, ds,\quad \forall t<0.
  \\ &\textrm{\Big(resp.}\ u\big(\gamma_+(t)\big)-u(x)=\int_0^t \Big[L_G\Big(\gamma_+(s),\dot \gamma_+(s),u\big(\gamma_+(s)\big)\Big)+c\Big]\, ds,\quad \forall t>0\textrm{\Big).}
\end{align*}
The curves satisfying the above equality are called $(u,L_G,c)$-calibrated curves.
\end{itemize}
\end{defn}

According to \cite[Appendix D]{NWY} and \cite[Appendix A]{n2}, we have
\begin{prop}\label{soleq}
The following are equivalent:
\begin{itemize}
\item[(i)] $u\prec L_G+c$.
\item[(ii)] $u(x)$ is a viscosity subsolution of \eqref{hjj}.
\item[(iii)] $u(x)$ is Lipschitz continuous and $G\big(x,Du(x),u(x)\big)\leqslant c$ holds almost everywhere.
\end{itemize}
The following are equivalent:
\begin{itemize}
\item[(i)] $u(x)$ is a viscosity solution of \eqref{hjj}.
\item[(ii)] $u(x)$ is a backward weak KAM solution of \eqref{hjj}.
\end{itemize}
\end{prop}

\subsection*{Aubry-Mather theory}

In what follows, we always assume that $H:T^*M\to\mathbb R$ is continuous and satisfies assumptions (H1) and (H2). We recall that $c_0$ denotes the critical value of $H$, and that $L:TM\to\mathbb R$ is the associated Lagrangian. We now collect several results from weak KAM theory; see \cite{Da4,gen,DZ10,FS}.

For $t>0$, let $h_t:M\times M\to\mathbb R$ be the minimal action function, which is defined as
\[h_t(x,y)=\inf_{\gamma}\int_0^t \big[L\big(\gamma(s),\dot\gamma(s)\big)+c_0\big]\, ds,\]
where the infimum is taken over all absolutely continuous curves $\gamma:[0,t]\to M$ satisfying $\gamma(0)=x$ and $\gamma(t)=y$. By Tonelli's theorem, the infimum is attained; see \cite{One}. The proof relies on the following lower semicontinuity result; see \cite[Theorem 3.5]{One}.
\begin{lem}\label{TM}
Let $J$ be a bounded interval of $\mathbb R$. Assume that $L(x,v)$ is lower semicontinuous, convex in $v$, and bounded from below. Then the integral functional
\begin{equation*}
  \mathcal L(\gamma):=\int_J L\big(\gamma(s),\dot \gamma(s)\big)\, ds
\end{equation*}
is sequentially weakly lower semicontinuous in $W^{1,1}(J,M)$, that is, if a sequence $(\gamma_n)_n$ converges weakly to $\gamma$ in $W^{1,1}(J,M)$, then
\[\mathcal L(\gamma)\leqslant\liminf_{n\to+\infty} \mathcal L(\gamma_n).\]
In particular, the above inequality applies if $(\gamma_n)_n$ converges uniformly to $\gamma$ and is equi-Lipschitz continuous.
\end{lem}
The Peierls barrier $h^\infty:M\times M\to\mathbb R$ is defined as
\begin{equation}\label{barr}
  h^\infty(x,y)=\liminf_{t\to+\infty}h_t(x,y).
\end{equation}
The projected Aubry set is defined as
\begin{equation}\label{pA}
  \mathcal A:=\{x\in M\mid h^\infty(x,x)=0\}.
\end{equation}

\begin{prop}\label{h>w}\cite[Proposition 3.6]{DZ10}.
The following properties hold
\begin{itemize}
\item[(i)] The Peierls barrier is finite-valued and Lipschitz continuous.
\item[(ii)] If $w$ is a subsolution of \eqref{E0}, then
\[h^\infty(x,y)\geqslant w(y)-w(x).\]
\item[(iii)] For each $x,y,z\in M$, the following triangle inequality holds
\[h^\infty(x,y)\leqslant h^\infty(x,z)+h^\infty(z,y).\]
\item[(iv)] The function $h^\infty(y,\cdot)$ gives a solution of \eqref{E0} for each $y\in M$. Similarly, the function $-h^\infty(\cdot,y)$ is a forward weak KAM solution and a viscosity subsolution of \eqref{E0} for each $y\in M$.
\end{itemize}
\end{prop}
\begin{prop}\label{prop AC}
If $u$ and $v$ are respectively a subsolution and a supersolution such that $u\leqslant  v$ on $\mathcal A$, then $u\leqslant  v$ on the whole of $M$. In particular, $\mathcal A$ is a uniqueness set for (\ref{E0}), meaning that if two solutions coincide on $\mathcal A$, then they are equal.
\end{prop}
Let $y\in\mathcal A$. The elementary solution associated with $y$ is given by $h^\infty(y,\cdot)$; see \cite[Section 4.2]{Ber}. By Proposition \ref{h>w} (ii), (iv), $h^\infty(y,\cdot)$ is the maximal subsolution $w$ of \eqref{E0} such that $w(y)=0$.

\begin{defn}
We say that a Borel probability measure $\tilde{\mu}$ on $TM$ is closed if
\begin{itemize}
\item [(1)] $\int_{TM}\|v\|_x\, d\tilde{\mu}(x,v)<+\infty$;
\item [(2)] for every function $f\in C^1(M)$, we have $\int_{TM}\langle Df,v\rangle_x\, d\tilde{\mu}(x,v)=0$.
\end{itemize}
\end{defn}

\begin{prop}\label{Mather}\cite[Theorem 5.7]{Da4}.
The following holds
\[\min_{\tilde{\mu}}\int_{TM}L(x,v)\, d\tilde{\mu}=-c_0,\]
where $\tilde{\mu}$ is taken among all closed measures on $TM$. Measures realizing the minimum are called Mather measures. We denote by $\widetilde{\mathfrak M}$ the set of all Mather measures, which is both convex and compact in the weak topology. The Mather set is defined as
\[\widetilde{\mathcal M}=\overline{\bigcup_{\tilde{\mu}\in\widetilde{\mathfrak M}}\supp(\tilde{\mu})}.\]
The projected Mather set is $\mathcal M=\pi(\widetilde{\mathcal M})$.
\end{prop}

\begin{prop}\cite[Proposition 3.13]{Z}.
$\mathcal M\subset \mathcal A$.
\end{prop}

Let $\tilde{\mu}$ be a Mather measure. The associated projected Mather measure $\mu$ is defined by
\[\int_M f(x)\, d\mu(x)=\int_{TM}f\big(\pi(x,v)\big)\, d\tilde{\mu}(x,v),\quad \forall f\in C(M).\]
Throughout this paper, we will denote by $\tilde\mu$ a probability measure on $TM$, and denote by $\mu$ a probability measure defined on $M$.

\subsection*{Results for contact H-J equations}

Finally, we collect several results from \cite{n2} that will be used in the proof of Theorem \ref{thm1}. Assume that $a\in C(M)$ and that there exist two points $x_1,x_2\in M$ such that $a(x_1)>0$ and $a(x_2)<0$. Let $\varphi\in C(M)$ and $c\in\mathbb R$. We define the solution semigroup $T_t^-:C(M)\to C(M)$ by
\begin{equation}\label{T-}
  T^-_t\varphi(x)=\inf_{\gamma(t)=x} \left\{\varphi\big(\gamma(0)\big)+\int_0^t\Big[L\big(\gamma(\tau),\dot{\gamma}(\tau)\big)+c- a\big(\gamma(\tau)\big)T^-_\tau\varphi\big(\gamma(\tau)\big)\Big]\, d\tau\right\},
\end{equation}
where the infimum is taken among absolutely continuous curves $\gamma:[0,t]\rightarrow M$ with $\gamma(t)=x$. Define the corresponding forward semigroup as
\begin{equation}\label{T+}
  T^+_t\varphi(x)=\sup_{\gamma(0)=x} \left\{\varphi\big(\gamma(t)\big)-\int_0^t\Big[L\big(\gamma(\tau),\dot{\gamma}(\tau)\big)+c- a\big(\gamma(\tau)\big)T^+_{t-\tau}\varphi\big(\gamma(\tau)\big)\Big]\, d\tau\right\}.
\end{equation}
\begin{lem}\label{n11} \cite[Lemma 5.4]{n2}.
If there is a strict subsolution $u_0$ of
\begin{equation}\label{hjj0}
  a(x)u(x)+H\big(x,Du(x)\big)=c\quad \textrm{in}\quad M,
\end{equation}
then the following limit exists
\[u_-=\lim_{t\to+\infty}T^-_t u_0,\]
and $u_-$ is the maximal solution of \eqref{hjj0}. Also, the following limit exists
\[v_+=\lim_{t\to+\infty}T^+_tu_0,\]
and $v_+$ is the minimal forward weak KAM solution of \eqref{hjj0}. Moreover, $v_+<u_-$.
\end{lem}

\begin{lem}\label{n12}\cite[Theorem 3]{n2}.
Assume that there exists a strict subsolution $u_0$ of \eqref{hjj0}. Let $u_-$ and $v_+$ denote the maximal viscosity solution and the minimal forward weak KAM solution of \eqref{hjj0}, respectively. If $\varphi\in C(M)$ satisfies $v_+<\varphi\leqslant u_-$, then $T_t^-\varphi$ converges uniformly to $u_-$ as $t\to+\infty$.
\end{lem}
Here we note that in the proof of \cite[Theorem 3]{n2}, the strict convexity of $H$ is not needed when $v_+<\varphi\leqslant u_-$.

\section{Proof of Theorem \ref{thm1}}

We first recall the setting considered in \cite{DNYZ}, where $a(x)$ was assumed to satisfy
\[
  \int_{TM} a(x)\, d\tilde{\mu} > 0
\]
for all Mather measures $\tilde{\mu}$ associated with $H$. In the present paper, however, we assume the integral of $a(x)$ with respect to Mather measures $\tilde\nu\in\widetilde{\mathfrak M}$ with $\supp(\nu)\cap \mathcal A_1=\emptyset$ to be negative, which introduces substantial new difficulties.

The first difficulty concerns the existence of solutions to \eqref{E}. In the non-monotone setting this is not automatic, and this is precisely the role of hypothesis {\rm (S)}. Replacing $\lambda_0$ by $\min\{\lambda_0,1\}$ if necessary, we assume throughout the proof that $\lambda_0\leqslant1$ and consider $\lambda\in(0,\lambda_0)$.

To prove the convergence of the maximal solution of \eqref{E}, we first rely on the large-time behavior described in Lemma \ref{n12}. This implies that any Mather measure $\tilde\mu_\lambda$ associated with the maximal solution of \eqref{E} satisfies
\[\int_{TM} a(x)\, d\tilde{\mu}_\lambda \geqslant 0,\]
as shown in Lemma \ref{mu*>0}. However, this information alone is not sufficient to conclude convergence. Let $u_*$ be a limit point of $\{u_\lambda\}_{\lambda\in(0,\lambda_0)}$. Then by Lemma \ref{leqA},
\[\int_{TM}a(x)u_*(x)\, d\tilde\mu\leqslant 0,\quad \forall \tilde\mu\in\widetilde{\mathfrak M}.\]
Let $\tilde\mu_z$ be a weak limit of $\tilde\mu_\lambda$. Then by Lemma \ref{auge} we have
\[\int_{TM} a(x)u_*(x)\, d\tilde{\mu}_z \geqslant 0.\]
Combining the above three inequalities, we will show that, if $u_*\neq u_0$, the image of $\mathcal A$ under $u_0-u_*$ contains a nontrivial interval. By $(\diamond)$ we conclude the convergence of $u_\lambda$ as $\lambda \to 0^+$.

\medskip

For each $\lambda\in(0,\lambda_0)$, let $v_\lambda$ be a strict subsolution given by hypothesis {\rm (S)}. Let $T^\lambda_t$ and $T^{\lambda,+}_t$ be the semigroups defined in \eqref{T-} and \eqref{T+}, associated with
\[L_\lambda(x,v,u):=L(x,v)+c_0-\lambda a(x)u,\]
respectively. By Lemma \ref{n11}, the limit
\[u_\lambda=\lim_{t\to+\infty}T^\lambda_t v_\lambda\]
exists, and $u_\lambda$ is the maximal solution of \eqref{E}. Also, the limit
\[v^+_\lambda=\lim_{t\to+\infty}T^{\lambda,+}_t v_\lambda\]
exists, and $v^+_\lambda$ is the minimal forward weak KAM solution of \eqref{E}, with $v^+_\lambda<u_\lambda$. Moreover, Lemma \ref{n12} describes the following large-time behavior, which will play a central role in the forthcoming proof. We state it again with new notations:

\begin{lem}\label{stab}
If $\varphi\in C(M)$ satisfies $v^+_\lambda<\varphi\leqslant u_\lambda$, then $T^\lambda_t \varphi$ converges uniformly to $u_\lambda$ as $t\to+\infty$.
\end{lem}

We first show that the family $\{u_\lambda\}_{\lambda\in(0,\lambda_0)}$ is uniformly bounded. Since $u_\lambda\geqslant v_\lambda\geqslant-C$, it suffices to establish a uniform upper bound for $u_\lambda$. Fix once and for all a viscosity solution $v_0$ of \eqref{E0}. As a first step, we prove that $u_\lambda$ is uniformly bounded from above on $\mathcal A_1$.
\begin{lem}
If $a(x)>0$ on $\mathcal A_1$, then $u_\lambda$ is bounded from above on $\mathcal A_1$ uniformly for $\lambda\in(0,\lambda_0)$.
\end{lem}
\begin{proof}
We claim that if
\[\max_{x\in\mathcal A_1}\big(u_\lambda(x)-v_0(x)\big)=u_\lambda(x_0)-v_0(x_0),\]
then $u_\lambda(x_0)\leqslant0$. Otherwise, assume $u_\lambda(x_0)>0$. By \cite[Section 4]{gen}, there is a curve $\gamma:\mathbb R\to M$ satisfying $\gamma(0)=x_0$ and $\gamma(\mathbb R)\subset\mathcal A$ which is $v_0$-calibrated. Since $\gamma$ is continuous and $\dist(\mathcal A_1,\mathcal A\setminus\mathcal A_1)\geqslant\delta>0$, we have $\gamma(\mathbb R)\subset\mathcal A_1$. By continuity, for $t_0>0$ small enough,
\[u_\lambda\big(\gamma(t)\big)>0\qquad\text{for all }t\in(-t_0,0).\]
Since $a>0$ on $\mathcal A_1$, for $t\in(0,t_0)$ we have
\begin{align*}u_\lambda(x_0)-u_\lambda\big(\gamma(-t)\big)
&\leqslant\int_{-t}^0\Big[L\big(\gamma(s),\dot\gamma(s)\big)+c_0-\lambda a\big(\gamma(s)\big)u_\lambda\big(\gamma(s)\big)\Big]ds
\\&<\int_{-t}^0\big[L\big(\gamma(s),\dot\gamma(s)\big)+c_0\big]ds
=v_0(x_0)-v_0\big(\gamma(-t)\big),
\end{align*}
which contradicts the maximality of $x_0$ for $u_\lambda-v_0$.

Therefore, for all $x\in\mathcal A_1$,
\[u_\lambda(x)-v_0(x)\leqslant u_\lambda(x_0)-v_0(x_0)\leqslant-v_0(x_0),
\]
and hence
\[u_\lambda(x)\leqslant2\|v_0\|_\infty\qquad\text{on }\mathcal A_1.
\]
\end{proof}

Using the fact that $u_\lambda\prec L_\lambda$, we can show that $u_\lambda$ is bounded from above on the whole space $M$, once it is bounded from above at one point in $M$.

\begin{lem}\label{bM}
The family $\{u_\lambda\}_{\lambda\in (0,\lambda_0)}$ is bounded from above on the whole $M$.
\end{lem}
\begin{proof}
Fix an arbitrary point $x\in M$. Let $y\in \mathcal A_1$. We take a geodesic $\alpha:[0,1]\to M$ with constant speed and connecting $y$ and $x$. We denote by $K_1>0$ the upper bound of $u_\lambda$ on $\mathcal A_1$. If $u_\lambda(x)\leqslant K_1$, then we have obtained the upper bound of $u_\lambda$. Thus, we only need to consider the case $u_\lambda(x)>K_1$. Then by continuity, there exists $\sigma\in [0,1)$ such that $u_\lambda\big(\alpha(\sigma)\big)=K_1$ and $u_\lambda\big(\alpha(s)\big)>K_1$ for all $s\in (\sigma,1]$. For $s\in (\sigma,1]$, we have
\begin{equation*}
\begin{aligned}
  &u_\lambda\big(\alpha(s)\big)-u_\lambda\big(\alpha(\sigma)\big)
  \\ &\leqslant \int_\sigma^s \Big[L\big(\alpha(\tau),\dot \alpha(\tau)\big)+c_0-\lambda a\big(\alpha(\tau)\big)u_\lambda\big(\alpha(\tau)\big)\Big]\, d\tau
  \\ &=\int_\sigma^s \Big[L\big(\alpha(\tau),\dot \alpha(\tau)\big)+c_0-\lambda a\big(\alpha(\tau)\big)\big(u_\lambda\big(\alpha(\tau)\big)-K_1\big)-\lambda a\big(\alpha(\tau)\big)K_1\Big]\, d\tau
  \\&\leqslant \max_{\substack{x\in M\\ \|v\|_x\leqslant \textrm{diam}(M)}}|L(x,v)+c_0|+\|a\|_\infty K_1+\lambda\|a\|_\infty\int_\sigma^s \big(u_\lambda\big(\alpha(\tau)\big)-K_1\big)\, d\tau.
\end{aligned}
\end{equation*}
Using the Gronwall inequality, we get
\[u_\lambda\big(\alpha(s)\big)-K_1\leqslant \Big(\max_{\substack{x\in M\\ \|v\|_x\leqslant \textrm{diam}(M)}}|L(x,v)+c_0|+\|a\|_\infty K_1\Big)e^{\|a\|_\infty},\quad \forall\lambda\in(0,\lambda_0).\]
Taking $s=1$, we get a uniform upper bound of $\{u_\lambda\}_{\lambda\in (0,\lambda_0)}$.
\end{proof}

\begin{lem}\label{s3}
The family $\{u_\lambda\}_{\lambda\in (0,\lambda_0)}$ is uniformly bounded and equi-Lipschitz continuous.
\end{lem}
\begin{proof}
We have already shown that $\{u_\lambda\}_{\lambda\in (0,\lambda_0)}$ is uniformly bounded. Now we are going to show that $\{u_\lambda\}_{\lambda\in (0,\lambda_0)}$ is equi-Lipschitz continuous. Fix $x,y\in M$. We take a geodesic $\alpha:[0,d(x,y)]\to M$ with constant speed connecting $x$ and $y$. It follows that
\begin{equation*}
\begin{aligned}
  u_\lambda(y)-u_\lambda(x)&\leqslant \int_0^{d(x,y)}\Big[L\big(\alpha(s),\dot{\alpha}(s)\big)+c_0-\lambda a\big(\alpha(s)\big) u_\lambda\big(\alpha(s)\big)\Big]\, ds
  \\ &\leqslant \Big(\max_{\substack{x\in M\\ \|v\|_x\leqslant 1}}|L(x,v)+c_0|+\|a\|_\infty\|u_\lambda\|_\infty\Big) d(x,y).
\end{aligned}
\end{equation*}
Exchanging $x$ and $y$, we can show that $\{u_\lambda\}_{\lambda\in (0,\lambda_0)}$ is equi-Lipschitz continuous.
\end{proof}

According to the Arzel\'a-Ascoli theorem and Lemma \ref{s3}, any sequence $(u_{\lambda_n})_n$ with $\lambda_n\to 0^+$ admits a subsequence which uniformly converges to a continuous function $u_*$. According to Proposition \ref{stability}, $u_*$ is a solution of \eqref{E0}.
\begin{lem}\label{leqA}
Let $u_*$ be a limit point of $\{u_\lambda\}_{\lambda\in(0,\lambda_0)}$. Then
\[\int_{TM}a(x)u_*(x)\,d\tilde\mu\leqslant0,
\qquad \forall\tilde\mu\in\widetilde{\mathfrak M}.\]
\end{lem}
\begin{proof}
By Proposition \ref{Mather}, $\int_{TM}L(x,v)\,d\tilde\mu=-c_0$ for all $\tilde\mu\in\widetilde{\mathfrak M}$. We note that $c_0$ is also the critical value of the Hamiltonian
\[H_\lambda(x,p):=H(x,p)+\lambda a(x)u_\lambda(x),\]
since $u_\lambda$ solves $H_\lambda(x,Du_\lambda)=c_0$. Therefore
\begin{align*}
\int_{TM}\big(L(x,v)-\lambda a(x)u_\lambda(x)\big)\,d\tilde\mu
&=-c_0-\lambda\int_{TM}a(x)u_\lambda(x)\,d\tilde\mu
\\&\geqslant\min_{\tilde\nu}\int_{TM}\big(L(x,v)-\lambda a(x)u_\lambda(x)\big)\,d\tilde\nu
=-c_0,
\end{align*}
where the minimum is taken among all closed probability measures on $TM$. Hence
\[\int_{TM}a(x)u_\lambda(x)\,d\tilde\mu\leqslant0.\]
Letting $\lambda\to0^+$ along a convergent subsequence gives the conclusion.
\end{proof}

Fix $z\in M$. Let $\gamma^z_\lambda:(-\infty,0]\to M$ be a $u_\lambda$-calibrated curve satisfying $\gamma^z_\lambda(0)=z$.
\begin{lem}\label{gameq}
The family $\{\gamma^z_\lambda\}_{z\in M,\ \lambda\in (0,\lambda_0)}$ is equi-Lipschitz continuous.
\end{lem}
\begin{proof}
By Lemma \ref{s3}, there are two constants $K,\bar \kappa>0$ independent of $\lambda$ such that $u_\lambda$ is bounded by $K$ and $u_\lambda$ is $\bar \kappa$-Lipschitz continuous. By the superlinearity of $L$, for each $T>0$, there is $C_T\in\mathbb R$ such that
\[L(x,v)+c_0\geqslant T\|v\|_x+C_T.\]
Thus, for $0\geqslant t>s$, we have
\begin{align*}
\bar \kappa d\big(\gamma^z_\lambda(t),\gamma^z_\lambda(s)\big)&\geqslant u_\lambda\big(\gamma^z_\lambda(t)\big)-u_\lambda\big(\gamma^z_\lambda(s)\big)
\\ &=\int_s^t \Big[L\big(\gamma^z_\lambda(\tau),\dot{\gamma}^z_\lambda(\tau)\big)+c_0-\lambda a\big(\gamma^z_\lambda(\tau)\big)u_\lambda\big(\gamma^z_\lambda(\tau)\big)\Big]\, d\tau
\\ &\geqslant \int_s^t \Big[(\bar \kappa+1)\|\dot{\gamma}^z_\lambda(\tau)\|_{\gamma^z_\lambda(\tau)}+C_{\bar \kappa+1}\Big]\, d\tau-\|a\|_\infty K(t-s)
\\ &\geqslant (\bar \kappa+1)d\big(\gamma^z_\lambda(t),\gamma^z_\lambda(s)\big)+(C_{\bar \kappa+1}-\|a\|_\infty K)(t-s),
\end{align*}
which implies that
\[d\big(\gamma^z_\lambda(t),\gamma^z_\lambda(s)\big)\leqslant (\|a\|_\infty K-C_{\bar \kappa+1})(t-s).\]
The proof is now complete.
\end{proof}

\begin{defn}\label{defmu}
Define measures $\tilde{\mu}^{z,\lambda}_{t}$ by
\begin{equation}\label{muz}
  \int_{TM} f(x,v)\, d\tilde{\mu}^{z,\lambda}_{t}:=\frac{1}{t}\int_{-t}^0f\big(\gamma^z_\lambda(s),\dot{\gamma}^z_\lambda(s)\big)\, ds,\quad \forall f\in C_c(TM).
\end{equation}
According to Lemma \ref{gameq}, the family $\{\tilde{\mu}^{z,\lambda}_{t}\}_{\substack{\lambda\in(0,\lambda_0)\\ t>0}}$ is tight. Hence, for each fixed $\lambda\in(0,\lambda_0)$, there exists a sequence $t_n\to +\infty$ such that $(\tilde{\mu}^{z,\lambda}_{t_n})_n$ weakly$^*$-converges to a probability measure $\tilde{\mu}_{z,\lambda}$. Moreover, there exists a sequence $\lambda_n\to 0^+$ such that $(\tilde{\mu}_{z,\lambda_n})_n$ weakly$^*$-converges to a probability measure $\tilde{\mu}_z$.
\end{defn}

\begin{lem}\label{mu*M}
Let $\tilde\mu_z$ be a weak$^*$-limit defined above. For each $z\in M$, the limit $\tilde{\mu}_z\in\widetilde{\mathfrak M}$.
\end{lem}
\begin{proof}
{\bf The measure $\tilde{\mu}_z$ is closed.} Let $f\in C^1(M)$. Then
\begin{align*}
\left|\int_{TM} \langle Df,v\rangle_x\, d\tilde{\mu}_{z,\lambda}\right|&=\left|\lim_{n\to +\infty}\frac{1}{t_n}\int_{-t_n}^0\frac{df}{ds}\big(\gamma^z_{\lambda}(s)\big)\, ds\right|
\\ &=\lim_{n\to +\infty}\frac{\big|f\big(\gamma^z_{\lambda}(0)\big)-f\big(\gamma^z_{\lambda}(-t_n)\big)\big|}{t_n}
\leqslant \lim_{n\to +\infty}\frac{2\|f\|_\infty}{t_n}=0.
\end{align*}
Therefore, $\tilde{\mu}_{z,\lambda}$ is closed. Then $\tilde{\mu}_z$ is also closed, since it is a limit of closed measures.

\medskip

{\bf The measure $\tilde{\mu}_z$ is minimizing.} Since $u_\lambda$ is uniformly bounded,
\begin{align*}
&\int_{TM} \big[L(x,v)+c_0-\lambda a(x)u_{\lambda}(x)\big]\, d\tilde{\mu}_{z,\lambda}
\\ &=\lim_{n\to +\infty}\frac{1}{t_n}\int_{-t_n}^0\Big[L\big(\gamma^z_{\lambda}(s),\dot{\gamma}^z_{\lambda}(s)\big)+c_0-\lambda a\big(\gamma^z_{\lambda}(s)\big)u_{\lambda}\big(\gamma^z_{\lambda}(s)\big)\Big]\, ds
\\ &=\lim_{n\to +\infty}\frac{u_{\lambda}\big(\gamma^z_{\lambda}(0)\big)-u_{\lambda}\big(\gamma^z_{\lambda}(-t_n)\big)}{t_n}=0.
\end{align*}
Let $(\tilde{\mu}_{z,\lambda_n})_n$ converge weakly$^*$ to
$\tilde{\mu}_z$. By the uniform bound on the velocities of the minimizing
curves, the measures $\tilde{\mu}_{z,\lambda_n}$ are supported in a common
compact subset of $TM$. Hence,
\[
\int_{TM} L(x,v)\,d\tilde{\mu}_z
=
\lim_{n\to+\infty}
\int_{TM} L(x,v)\,d\tilde{\mu}_{z,\lambda_n}.
\]
Since $u_{\lambda_n}$ is uniformly bounded, we have
\[
\int_{TM} L(x,v)\,d\tilde{\mu}_z
=
\lim_{n\to+\infty}\int_{TM}
\big[L(x,v)-\lambda_n a(x)u_{\lambda_n}(x)\big]
\,d\tilde{\mu}_{z,\lambda_n}
=-c_0,
\]
which implies that $\tilde{\mu}_z$ is minimizing.
\end{proof}

\begin{lem}\label{mu*>0}
For each $z\in M$ and $\lambda\in(0,\lambda_0)$, we have $\int_{TM} a(x)\, d\tilde{\mu}_{z,\lambda}\geqslant 0$.
\end{lem}
\begin{proof}
We argue by contradiction. Assume there is $z\in M$, $\lambda\in(0,\lambda_0)$ and $B_{z,\lambda}>0$ such that
\[\int_{TM} a(x)\, d\tilde{\mu}_{z,\lambda}<-B_{z,\lambda}.\]
Then there is a sequence $t_n\to+\infty$, and $N>0$ such that if $n\geqslant N$, we have
\[\frac{1}{t_n}\int_{-t_n}^0a\big(\gamma^z_{\lambda}(s)\big)\, ds<-B_{z,\lambda}.\]
We take a small constant $\varepsilon>0$ such that
\[0<\varepsilon<\min_{x\in M}(u_\lambda-v^+_\lambda),\]
and define $u_\varepsilon:=u_{\lambda}-\varepsilon$. By Lemma \ref{stab}, there is $t_{\lambda}>0$ depending on $\lambda$ such that
\begin{equation}\label{T>del}
  u_{\lambda}(z)-T^{\lambda}_tu_\varepsilon(z)<\varepsilon,\quad \forall t>t_{\lambda}.
\end{equation}
Note that, since the convergence in Lemma \ref{stab} is uniform, $t_\lambda$ is independent of $z$. Define $s_n:=t_n+t_{\lambda}$, and
\[w_\lambda(s):=u_{\lambda}\big(\gamma^z_\lambda(s)\big)-T^{\lambda}_{s+s_n}u_\varepsilon\big(\gamma^z_\lambda(s)\big),\quad s\in(-s_n,0).\]
Since $\gamma^z_\lambda$ is a $u_\lambda$-calibrated curve, we have
\[\frac{du_{\lambda}}{ds}\big(\gamma^z_\lambda(s)\big)=L\big(\gamma^z_\lambda(s),\dot{\gamma}^z_\lambda(s)\big)+c_0-\lambda a\big(\gamma^z_\lambda(s)\big)u_{\lambda}\big(\gamma^z_\lambda(s)\big),\quad a.e.\ s<0.\]
Moreover, by the definition of the solution semigroup, it follows that
\begin{align*}
\frac{dT^{\lambda}_{s+s_n}u_\varepsilon}{ds}\big(\gamma^z_\lambda(s)\big)
\leqslant &L\big(\gamma^z_\lambda(s),\dot{\gamma}^z_\lambda(s)\big)+c_0
\\ &-\lambda a\big(\gamma^z_\lambda(s)\big)T^{\lambda}_{s+s_n}u_\varepsilon\big(\gamma^z_\lambda(s)\big),\quad a.e.\ s\in(-s_n,0).
\end{align*}
We conclude that
\[\dot w_\lambda(s)\geqslant -\lambda a\big(\gamma^z_\lambda(s)\big)w_\lambda(s),\quad a.e.\ s\in(-s_n,0).\]
Notice that
\[w_\lambda(-s_n)=u_{\lambda}\big(\gamma^z_\lambda(-s_n)\big)-u_\varepsilon\big(\gamma^z_\lambda(-s_n)\big)=\varepsilon.\]
When $t_n$ is large enough, we have
\begin{align*}
w_{\lambda}(0)&=u_{\lambda}(z)-T^{\lambda}_{s_n}u_\varepsilon(z)
\\ &\geqslant \varepsilon e^{-\int_{-t_n-t_{\lambda}}^{-t_n}\lambda a\textrm{$\big(\gamma^z_\lambda(s)\big)$}\, ds}e^{-\int_{-t_n}^{0}\lambda a\textrm{$\big(\gamma^z_\lambda(s)\big)$}\, ds}>\varepsilon e^{-\lambda \|a\|_\infty t_{\lambda}}e^{\lambda B_{z,\lambda} t_n}>\varepsilon,
\end{align*}
which contradicts \eqref{T>del}.
\end{proof}

\begin{lem}\label{auge}
Assume $u_{\lambda_n}$ converges to $u_*$ uniformly. For each $z\in M$, let $\tilde\mu_z$ be a weak$^*$-limit of $(\tilde \mu_{z,\lambda_n})_n$ defined above. Then
\[
\int_{TM}a(x)u_*(x)\,d\tilde\mu_z\geqslant0.
\]
\end{lem}
\begin{proof}
For each $\lambda\in(0,\lambda_0)$, assume $\tilde \mu^{z,\lambda}_{t_n}\to\tilde \mu_{z,\lambda}$ weakly. By \eqref{muz},
\begin{align*}
&\int_{TM}\big[L(x,v)+c_0-\lambda a(x)u_{\lambda}(x)\big] \, d\tilde{\mu}_{z,\lambda}
\\&=\lim_{n\to+\infty}\frac{1}{t_n}\int_{-t_n}^0\Big[L\big(\gamma^z_{\lambda}(s),\dot{\gamma}^z_{\lambda}(s)\big)+c_0-\lambda a\big(\gamma^z_{\lambda}(s)\big)u_{\lambda}\big(\gamma^z_{\lambda}(s)\big)\Big]ds
\\&=\lim_{n\to+\infty}\frac{u_{\lambda}\big(\gamma^z_{\lambda}(0)\big)-u_{\lambda}\big(\gamma^z_{\lambda}(-t_n)\big)}{t_n}=0.
\end{align*}
Here we recall that $\tilde{\mu}_{z,\lambda}$ is closed, which implies
\[\int_{TM}L(x,v)\,d\tilde\mu_{z,\lambda}\geqslant-c_0.\]
We obtain
\[\lambda\int_{TM}a(x)u_\lambda(x)\,d\tilde\mu_{z,\lambda}
=\int_{TM}[L(x,v)+c_0] \,d\tilde\mu_{z,\lambda}\geqslant0.\]
Let $\lambda_n\to 0$, and up to a subsequence, we obtain the desired inequality.
\end{proof}

Any extremal decomposition of $\tilde\mu\in\widetilde{\mathfrak M}_1$ is concentrated on extremal Mather measures $\tilde\nu$ satisfying $\supp(\nu)\subset\supp(\mu)\subset\mathcal A_1$. Hence, $a(x)>0$ on $\supp(\nu)$ for such measures.

\begin{lem}\label{extm}
For each extremal point $\tilde\mu\in\widetilde{\mathfrak M}$, we have either $\supp(\mu)\subset \mathcal A_1$ or $\supp(\mu)\cap \mathcal A_1=\emptyset$.
\end{lem}
\begin{proof}
Let $\tilde\mu$ be an extremal Mather measure. Assume by contradiction that $\supp(\mu)\cap\mathcal A_1\neq\emptyset$ and $\supp(\mu)\cap(\mathcal A\setminus\mathcal A_1)\neq\emptyset$. Since $\dist(\mathcal A_1,\mathcal A\setminus \mathcal A_1)\geqslant \delta>0$, we can take $f\in C^1(M)$ such that $f\equiv 1$ near $\mathcal A_1$, $f\equiv 0$ near $\mathcal A\setminus \mathcal A_1$ and $\alpha:=\int_{TM} f(x)\, d\tilde\mu\in (0,1)$. In this case, $Df=0$ on $\mathcal A$. Then for all $\phi\in C^1(M)$,
\[\int_{TM}\langle D\phi(x),v\rangle_x f(x)\, d\tilde\mu=\int_{TM}\langle D(f\phi)(x),v\rangle_x \, d\tilde\mu=0,\]
which implies that $\tilde\mu_1:=\frac{f}{\alpha}\tilde\mu$ is a closed probability measure. Similarly, $\tilde\mu_2:=\frac{1-f}{1-\alpha}\tilde\mu$ is also a closed probability measure. By Proposition \ref{Mather}, we have
\[\int_{TM}L(x,v)\, d\tilde\mu_1\geqslant -c_0\quad\text{and}\quad \int_{TM}L(x,v)\, d\tilde\mu_2\geqslant -c_0.\]
We also have
\[\int_{TM}L(x,v)\, d\tilde\mu=\alpha\int_{TM}L(x,v)\, d\tilde\mu_1+(1-\alpha)\int_{TM}L(x,v)\, d\tilde\mu_2=-c_0,\]
which implies that
\[\int_{TM}L(x,v)\, d\tilde\mu_1= -c_0\quad\text{and}\quad \int_{TM}L(x,v)\, d\tilde\mu_2= -c_0.\]
Hence, both $\tilde\mu_1$ and $\tilde\mu_2$ are Mather measures, which contradicts the extremality of $\tilde\mu$.
\end{proof}

\begin{lem}\label{w1w2}
For any two subsolutions $w_1$ and $w_2$ of \eqref{E0}, $w_1-w_2$ is a constant on the support of $\mu$, where $\mu$ is the projection of an extreme point $\tilde\mu$ of $\widetilde{\mathfrak M}$.
\end{lem}

\begin{proof}
Set $f:=w_1-w_2$.

\medskip

{\bf Step 1.} By Proposition \ref{ue}, for $i=1,2$ there exist smooth functions $w_i^\varepsilon\in C^\infty(M)$ such that
\[
w_i^\varepsilon \to w_i \quad\text{uniformly on }M,
\]
and
\[
H\big(x,Dw_i^\varepsilon(x)\big)\leqslant c_0+\varepsilon \qquad \forall x\in M.
\]
Define $f^\varepsilon:=w_1^\varepsilon-w_2^\varepsilon$. Then $f^\varepsilon\to f$ uniformly.

\medskip

For each $i$, Fenchel's inequality gives, for all $(x,v)\in TM$,
\[
\langle Dw_i^\varepsilon(x),v\rangle_x  \leqslant L(x,v)+H\big(x,Dw_i^\varepsilon(x)\big)
\leqslant L(x,v)+c_0+\varepsilon.
\]
Since $\tilde \mu$ is closed, $\int_{TM} \langle Dw_i^\varepsilon(x), v\rangle_x \,d\tilde \mu=0$. Also, since $\tilde \mu\in\widetilde{\mathfrak M}$,
\[
\int_{TM} \big(L(x,v)+c_0\big)\,d\tilde\mu=0.
\]
Therefore,
\[L(x,v)+c_0-\langle Dw_i^\varepsilon(x), v\rangle_x \geqslant-\varepsilon\quad \text{and}\quad \int_{TM}\big(L(x,v)+c_0-\langle Dw_i^\varepsilon(x), v\rangle_x \big)\,d\tilde\mu=0.\]
Hence,
\[
\int_{TM} \big|L(x,v)+c_0-\langle Dw_i^\varepsilon(x),v\rangle_x \big|\,d\tilde\mu \leqslant 2\varepsilon.
\]
Subtracting the inequalities for $i=1$ and $i=2$ yields
\begin{equation}\label{4ep}
\int_{TM} |\langle Df^\varepsilon(x), v\rangle_x |\,d\tilde\mu \leqslant 4\varepsilon.
\end{equation}

\medskip

{\bf Step 2.} Let $\theta\in C^1_b(\mathbb R)$ be bounded with bounded derivative. For any $\phi\in C^1(M)$, the function $x\mapsto \phi(x)\theta\big(f^\varepsilon(x)\big)$ is smooth enough to be used as a test function in the closedness condition for $\tilde\mu$, that is, we have
\[
\int_{TM} \big\langle D\big(\phi\,\theta(f^\varepsilon)\big)(x),v\big\rangle_x\,d\tilde \mu=0.
\]
Expanding,
\[
\int_{TM}\theta(f^\varepsilon)\langle D\phi,v\rangle_x\,d\tilde \mu
+\int_{TM}\phi\,\theta'(f^\varepsilon)\langle Df^\varepsilon,v\rangle_x\,d\tilde\mu=0.
\]
The second integral is bounded by $\|\phi\|_\infty\|\theta'\|_\infty \int_{TM}|\langle Df^\varepsilon,v\rangle_x|\,d\tilde \mu$, which tends to $0$ as $\varepsilon \to 0$ by \eqref{4ep}. The first integral converges to
\[
\int_{TM}\theta(f)\langle D\phi,v\rangle_x\,d\tilde\mu
\]
because $f^\varepsilon\to f$ uniformly and $\tilde \mu$ has finite first moment. Hence
\[
\int_{TM}\theta(f)\langle D\phi,v\rangle_x\,d\tilde \mu=0.
\]
Thus $\theta(f)\tilde\mu$ is a closed signed measure. Since $\theta(f)$ is bounded, the total variation is finite.

\medskip

{\bf Step 3.} Assume, for contradiction, that $f$ is not constant $\mu$-almost everywhere. Choose a strictly increasing function $\theta\in C^1(\mathbb R)$ with $\theta(s)\in (0,1)$ for all $s\in\mathbb R$ and $x\mapsto \theta\big(f(x)\big)$ is not constant $\mu$-a.e. Define
\[
\alpha:=\int_{TM}\theta\big(f(x)\big)\,d\tilde\mu \in (0,1),
\]
and set
\[
\tilde\mu_1:=\frac{\theta(f)}{\alpha}\tilde\mu,\qquad 
\tilde\mu_2:=\frac{1-\theta(f)}{1-\alpha}\tilde\mu.
\]
By Step 2, both $\tilde\mu_1$ and $\tilde\mu_2$ are closed measures. Moreover, they are minimizing: since every closed measure $\tilde\nu$ satisfies $\int_{TM} L(x,v)\,d\tilde\nu\geqslant -c_0$, and
\[
-c_0=\int_{TM} L(x,v)\,d\tilde\mu
=\alpha\int_{TM} L(x,v)\,d\tilde\mu_1+(1-\alpha)\int_{TM} L(x,v)\,d\tilde\mu_2,
\]
we must have equality in both terms, i.e.
\[
\int_{TM} L(x,v)\,d\tilde \mu_1=\int_{TM} L(x,v)\,d\tilde\mu_2=-c_0.
\]
Thus $\tilde\mu_1,\tilde\mu_2\in\widetilde{\mathfrak M}$. They are distinct because $\theta(f)$ is not constant $\mu$-a.e. The convex decomposition
\[
\tilde\mu=\alpha \tilde \mu_1+(1-\alpha)\tilde\mu_2
\]
is therefore nontrivial, contradicting the extremality of $\tilde\mu$.

Finally, since $w_1-w_2$ is continuous, it is constant on $\supp(\mu)$.
\end{proof}

Let $u_{\lambda_n}\to u_*$. By Lemma \ref{leqA}, $u_*\leqslant u_0$.
\begin{lem}\label{u0<A}
We have
\[\int_{TM}a(x)u_0(x)\,d\tilde\mu\leqslant0
\quad\text{for all }\tilde\mu\in\widetilde{\mathfrak M}_1.\]
\end{lem}
\begin{proof}
Notice first that $\mathcal S\neq\emptyset$. Indeed, if $u_{\lambda_n}\to u_*$ uniformly, then Lemma \ref{leqA} gives $u_*\in\mathcal S$.
Fix a solution $v$ of \eqref{E0}. By Lemma \ref{w1w2}, for each $w\in\mathcal S$ and each extremal measure $\tilde\mu\in\widetilde{\mathfrak M}_1$, there exists a constant $c(w,\tilde\mu)\in\R$ such that
\[w(x)=v(x)+c(w,\tilde\mu)\quad\text{for all }x\in\supp(\mu).\]
Since $w\in\mathcal S$,
\[c(w,\tilde\mu)\leqslant
-\frac{\int_{TM}a(x)v(x)\,d\tilde\mu}
{\int_{TM}a(x)\,d\tilde\mu}.\]
Fixing one such extremal measure $\tilde\mu$, the right-hand side is independent of $w$. Hence $\mathcal S$ is uniformly bounded from above on $\supp(\mu)$. Since critical subsolutions are equi-Lipschitz by the coercivity of $H$, $\mathcal S$ is uniformly bounded from above on $M$. Therefore $u_0:=\sup_{w\in\mathcal S}w$ is finite and Lipschitz continuous. By Proposition \ref{supsub}, $u_0$ is a subsolution of \eqref{E0}. Moreover, setting $c(\tilde\mu):=\sup_{w\in\mathcal S}c(w,\tilde\mu)$, we have
\[u_0(x)=v(x)+c(\tilde\mu)\quad\text{for all }x\in\supp(\mu)\text{ and all extremal measures }\tilde\mu\in\widetilde{\mathfrak M}_1,\]
and
\[c(\tilde\mu)\leqslant
-\frac{\int_{TM}a(x)v(x)\,d\tilde\mu}
{\int_{TM}a(x)\,d\tilde\mu}.\]
Hence,
\begin{align*}
&\int_{TM}a(x)u_0(x)\, d\tilde \mu
\\ &=\int_{TM}a(x)v(x)\, d\tilde \mu+c(\tilde\mu)\int_{TM}a(x)\, d\tilde \mu\leqslant 0\quad \text{for all extremal measures }\tilde \mu\in\widetilde{\mathfrak M}_1,
\end{align*}
which implies
\[\int_{TM}a(x)u_0(x)\, d\tilde \mu\leqslant 0\quad \text{for all }\tilde \mu\in\widetilde{\mathfrak M}_1.\]
\end{proof}

\begin{lem}\label{ganga}
Let $u_{\lambda_n}$ converge to $u_*$ uniformly. Then up to a subsequence, $\gamma_n:=\gamma^z_{\lambda_n}$ converges to $\gamma$ locally uniformly and $\gamma$ is a $u_*$-calibrated curve.
\end{lem}
\begin{proof}
Fix $T>0$. By the Arzel\'a-Ascoli theorem and Lemma \ref{gameq}, up to a subsequence, $\gamma_n\to \gamma$ uniformly on $[-T,0]$. By Lemma \ref{TM}, for $-T\leqslant s<t\leqslant 0$, we have
\begin{align*}
u_*\big(\gamma(t)\big)-u_*\big(\gamma(s)\big)&=\lim_{n\to+\infty}\Big(u_{\lambda_n}\big(\gamma_n(t)\big)-u_{\lambda_n}\big(\gamma_n(s)\big)\Big)
\\ &=\lim_{n\to+\infty}\int_s^t \Big[L\big(\gamma_n(\tau),\dot\gamma_n(\tau)\big)+c_0-\lambda_n a\big(\gamma_n(\tau)\big)u_{\lambda_n}\big(\gamma_n(\tau)\big)\Big]\, d\tau
\\ &\geqslant \int_s^t \big[L\big(\gamma(\tau),\dot\gamma(\tau)\big)+c_0\big]\, d\tau.
\end{align*}
On the other hand, since $u_*\prec L+c_0$, one obtains the equality. By a diagonal argument, the subsequence can be chosen independently of $T$.
\end{proof}

\begin{lem}\label{u0-u*}
Let $I\subset \R$ and $\gamma:I\to M$ be a $u_*$-calibrated curve. Then the function $t\mapsto (u_0-u_*)\big(\gamma(t)\big)$ is non-increasing.
\end{lem}
\begin{proof}
Fix $a,b\in I$ with $a<b$. Since $u_0$ is a subsolution, we have
\[u_0\big(\gamma(b)\big)-u_0\big(\gamma(a)\big)\leqslant \int_a^b \big[L\big(\gamma(s),\dot\gamma(s)\big)+c_0\big]\, ds.\]
Since $\gamma$ is $u_*$-calibrated,
\[u_*\big(\gamma(b)\big)-u_*\big(\gamma(a)\big)=\int_a^b \big[L\big(\gamma(s),\dot\gamma(s)\big)+c_0\big]\, ds.\]
Subtracting the equality from the inequality, we obtain
\[u_0\big(\gamma(b)\big)-u_*\big(\gamma(b)\big)\leqslant u_0\big(\gamma(a)\big)-u_*(\gamma(a)\big),\]
which completes the proof.
\end{proof}

\begin{lem}\label{gminA}
Let $\gamma:(-\infty,+\infty)\to M$ be a $u_*$-calibrated curve. Then if there is a sequence $t_n\to+\infty$ such that $\gamma(t_n)\to x$, then $x\in \mathcal A$.
\end{lem}
\begin{proof}
Assume $t_n\to+\infty$ such that $\gamma(t_n)\to x$. Up to a subsequence, one may assume $t_{n+1}-t_n\to+\infty$. We denote $d_n:=d\big(x,\gamma(t_n)\big)$. Let $\alpha_1:[0,d_n]\to M$ and $\alpha_2:[0,d_{n+1}]\to M$ be two geodesics with constant speed $1$ and $\alpha_1(0)=x$, $\alpha_1(d_n)=\gamma(t_n)$, $\alpha_2(0)=\gamma(t_{n+1})$, and $\alpha_2(d_{n+1})=x$. Define
\begin{equation*}
\gamma_n(s):=
\begin{cases}
\alpha_1(s) &\text{for} \ 0\leqslant s\leqslant d_n,
\\ \gamma(s+t_n-d_n) & \text{for} \ d_n\leqslant s\leqslant t_{n+1}-t_n+d_n,
\\ \alpha_2(s-t_{n+1}+t_n-d_n) &\text{for} \ t_{n+1}-t_n+d_n\leqslant s\le t_{n+1}-t_n+d_n+d_{n+1}.
\end{cases}
\end{equation*}
Then
\begin{align*}
&h_{t_{n+1}-t_n+d_n+d_{n+1}}(x,x)\leqslant \int_0^{t_{n+1}-t_n+d_n+d_{n+1}}\big[L\big(\gamma_n(s),\dot\gamma_n(s)\big)+c_0\big]\, ds
\\ &=\int_0^{d_n}\big[L\big(\alpha_1(s),\dot\alpha_1(s)\big)+c_0\big]\, ds
\\ &\qquad +\int_{t_n}^{t_{n+1}}\big[L\big(\gamma(s),\dot\gamma(s)\big)+c_0\big]\, ds+\int_0^{d_{n+1}}\big[L\big(\alpha_2(s),\dot\alpha_2(s)\big)+c_0\big]\, ds
\\ &\leqslant \max_{x\in M,\ \|v\|_x\leqslant 1}|L(x,v)+c_0|(d_n+d_{n+1})+u_*\big(\gamma(t_{n+1})\big)-u_*\big(\gamma(t_n)\big).
\end{align*}
Since $d_n\to 0$ and $u_*$ is continuous, the right hand side of the above inequality tends to zero as $n\to +\infty$. Hence,
\[h^\infty(x,x)\leqslant \liminf_{n\to+\infty}h_{t_{n+1}-t_n+d_n+d_{n+1}}(x,x)\leqslant 0,\]
which implies that $x\in\mathcal A$ since $h^\infty(x,x)\geqslant 0$.
\end{proof}

Assume that $u_*\neq u_0$. Then there exists $x\in M$ such that
\[m:=u_0(x)-u_*(x)>0.\]
Take $r\in (0,m)$. By continuity, the set
\[V_r=\{y\in M\mid u_0(y)-u_*(y)>r\}\]
is an open neighborhood of $x$. Moreover,
\[u_0(z)-u_*(z)=r\quad \text{for all }z\in \partial V_r.\]
Let $\gamma^x_\lambda$ be as defined before, and we define $\gamma_n:=\gamma^x_{\lambda_n}$. For each $\lambda_n$, we define
\[\ell_n:=\max\{s\in (-\infty,0]\mid \gamma_n(s)\in \partial V_r\}.\]
If $\gamma_n\big((-\infty,0]\big)\cap \partial V_r=\emptyset$, we set $\ell_n=-\infty$. Since $\gamma_n(0)=x\in V_r$, we have $\ell_n<0$ and $\gamma_n\big((\ell_n,0]\big)\subset V_r$. Moreover, if $\ell_n>-\infty$, then $\gamma_n(\ell_n)\in \partial V_r$.

\medskip

\noindent {\bf Claim 1.} For $\lambda _n$ small enough, we have $\ell_n>-\infty$.

\medskip

We argue by contradiction. Assume that, up to a subsequence, $\ell_n=-\infty$, that is, $\gamma_n\big((-\infty,0]\big)\subset V_r$. We define the limit measure $\tilde\mu_x$ associated with $\gamma_n$ as in Definition \ref{defmu}. By Lemma \ref{auge}, we have
\begin{equation}\label{u*>}
\int_{TM}a(x)u_*(x)\, d\tilde\mu_x\geqslant 0.
\end{equation}
By Lemma \ref{mu*>0},
\[\int_{TM}a(x)\, d\tilde\mu_x\geqslant 0.\]
Since $\gamma_n\big((-\infty,0]\big)\subset V_r$, we have $\supp(\mu_x)\subset\overline V_r$.
Consequently, any extremal decomposition of $\tilde\mu_x$
is concentrated on extremal Mather measures whose projected supports
are contained in $\overline V_r$. By assumption (A) and Lemma \ref{extm}, there exists an extremal Mather measure $\tilde\mu\in \widetilde{\mathfrak M}_1$ with $\supp(\mu)\subset \ol{V}_r$ in the extremal decomposition of $\tilde\mu_x$. Since $u_0-u_*\geqslant r$ and $a(x)>0$ on $\supp(\mu)$, we get
\[\int_{TM}a(x)\big(u_0(x)-u_*(x)\big)\, d\tilde\mu\geqslant r\int_{TM}a(x)\, d\tilde\mu\geqslant r\min_{x\in \supp(\mu)}a(x).\]
By \eqref{u*>}, the extremal decomposition of $\tilde\mu_x$, and the fact that
\[\int_{TM}a(x)u_*(x)\, d\tilde\nu\leqslant 0\]
for all Mather measures $\tilde\nu$, the extremal Mather measure $\tilde\mu$ must satisfy
\[\int_{TM}a(x)u_*(x)\, d\tilde\mu = 0.\]
Hence,
\[\int_{TM}a(x)u_0(x)\, d\tilde\mu\geqslant r\min_{x\in \supp(\mu)}a(x)>0,\]
which contradicts Lemma \ref{u0<A}.

\medskip

\noindent {\bf Claim 2.} $\ell_n \to -\infty$ as $\lambda_n \to 0$.

\medskip

We argue by contradiction. Assume that, up to a subsequence, $\ell_n$ is bounded away from $-\infty$. Extracting further, we can assume $\ell_n\to \ell_*$. By Lemma \ref{ganga}, we know that $\gamma_n \to\gamma$ locally uniformly, and that $\gamma$ is a $u_*$-calibrated curve. Moreover, $\gamma(\ell_*)\in \partial V_r$. By Lemma \ref{u0-u*},
\[r=u_0\big(\gamma(\ell_*)\big)-u_*\big(\gamma(\ell_*)\big)\geqslant u_0\big(\gamma(0)\big)-u_*\big(\gamma(0)\big),\]
which contradicts $r<m$.

\medskip

By Claim 1, up to a subsequence, $\ell_n>-\infty$. We define $\eta_n:(-\infty,-\ell_n]\to M$ as
\[\eta_n(t)=\gamma_n(t+\ell_n),\quad t\in (-\infty,-\ell_n].\]
By Claim 2 and Lemma \ref{ganga}, by a diagonal extraction, up to a subsequence, $\eta_n \to \eta$ locally uniformly, where $\eta:(-\infty, +\infty)\to M$ is a $u_*$-calibrated curve. Moreover, since $\eta_n(0)\in \partial V_r$ and $\eta_n([0,-\ell_n])\subset \ol V_r$, we have
\[\eta(0)\in \partial V_r,\quad \eta\big([0,+\infty)\big)\subset \ol V_r.\]
Since $u_0-u_*$ is non-increasing along $\eta$, we have
\[u_0\big(\eta(t)\big)-u_*\big(\eta(t)\big)=r\quad \text{for all }t>0.\]
By Lemma \ref{gminA}, we have $r\in (u_0-u_*)(\mathcal A)$.
\begin{lem}
For any two subsolutions $v_1,v_2$ of \eqref{E0} and every $x,y\in M$, we have
\[|(v_1-v_2)(y)-(v_1-v_2)(x)|\leqslant d_H(x,y),\]
that is, $v_1-v_2$ defines a Lipschitz continuous function on the Mather quotient with respect to $d_H$, with Lipschitz constant $\leqslant 1$.
\end{lem}
\begin{proof}
By Proposition \ref{h>w}, we have
\[v_1(y)-v_1(x)\leqslant h^\infty(x,y),\]
and
\[v_2(x)-v_2(y)\leqslant h^\infty(y,x).\]
Adding and rearranging yields
\[(v_1-v_2)(y)-(v_1-v_2)(x)\leqslant h^\infty(x,y)+h^\infty(y,x).\]
Since $d_H$ is symmetric, we get
\[|(v_1-v_2)(y)-(v_1-v_2)(x)|\leqslant d_H(x,y).\]
\end{proof}

Now, we have proved that $(0,m)\subset (u_0-u_*)(\mathcal A)$, and $u_0-u_*$ defines a Lipschitz continuous function on the Mather quotient. Therefore, if the Mather quotient set has 1-dimensional Hausdorff measure zero, we have $m=0$, i.e., $u_0=u_*$. This finishes the proof of the convergence.

\medskip

We define
\[\hat u_0(x):=\inf_{\tilde{\mu}\in\widetilde{\mathfrak M}_1}\frac{\int_{TM} h^\infty(y,x)a(y)\, d\tilde\mu(y,v)}{\int_{TM} a(y)\, d\tilde\mu(y,v)},\qquad x\in M.\]
\begin{lem}
$u_0=\hat u_0$.
\end{lem}
\begin{proof}
We first note that $\hat u_0$ is a subsolution of \eqref{E0}.
For each $\tilde\mu\in\widetilde{\mathfrak M}_1$, set
\[
F_{\tilde\mu}(x):=
\frac{\int_{TM}h^\infty(y,x)a(y)\,d\tilde\mu(y,v)}
{\int_{TM}a(y)\,d\tilde\mu(y,v)}.
\]
Since $a>0$ on $\supp(\mu)$, we can write
\[
F_{\tilde\mu}(x)
=
\int_{TM}h^\infty(y,x)\,d\rho_{\tilde\mu}(y,v)
\]
where
\[
d\rho_{\tilde\mu}(y,v)
:=
\frac{a(y)}{\int_{TM}a\,d\tilde\mu}\,d\tilde\mu(y,v)
\]
is a probability measure. By Proposition \ref{h>w}, for every $y\in M$,
$h^\infty(y,\cdot)$ is a subsolution of \eqref{E0}. Hence, by the
convexity of $H$ and \cite[Proposition 2.1(iii)]{Da4},
$F_{\tilde\mu}$ is a subsolution of \eqref{E0}.
Moreover, the family
$\{F_{\tilde\mu}\}_{\tilde\mu\in\widetilde{\mathfrak M}_1}$
is equi-Lipschitz, since $h^\infty$ is Lipschitz continuous.
Therefore, by \cite[Proposition 2.1(ii)]{Da4},
\[
\hat u_0=\inf_{\tilde\mu\in\widetilde{\mathfrak M}_1}F_{\tilde\mu}
\]
is a subsolution of \eqref{E0}.

Let us now prove that $u_0\leqslant  \hat u_0$. By Proposition \ref{h>w}, we know that 
\[u_0(x)\leqslant  u_0(y)+h^\infty(y,x)\quad\text{for all }x,y\in M.\]
Let us integrate this inequality with respect to 
$\tilde{\mu}\in\widetilde{\mathfrak M}_1$. We infer 
\begin{align*}
u_0(x)\int_{TM}a(y)\,d\tilde\mu(y,v)
&\leqslant  \int_{TM}u_0(y)a(y)\,d\tilde\mu(y,v)+\int_{TM}h^\infty(y,x)a(y)\,d\tilde\mu(y,v)
\\ &\leqslant \int_{TM}h^\infty(y,x)a(y)\,d\tilde\mu(y,v),
\end{align*}
where, for the last inequality, we used Lemma \ref{u0<A}. From this we get 
\[
u_0(x)
\leqslant 
\dfrac{\int_{TM}h^\infty(y,x)a(y)\,d\tilde\mu(y,v)}{\int_{TM}a(y)\,d\tilde\mu(y,v)}
\qquad
\hbox{for all $x\in M$.}
\]
By taking the inf with respect to $\tilde\mu\in{\widetilde{\mathfrak M}_1}$ of the right-hand side term in the above inequality, we get $u_0\leqslant  \hat u_0$. 

Finally, we prove that $u_0\geqslant  \hat u_0$. For every fixed $z\in \mathcal A$, set 
\[
U_{z}(x):=-h^\infty(x,z)+\hat u_0(z), \qquad x\in M.
\]
Then $U_z$ is a subsolution of (\ref{E0}), by Proposition \ref{h>w}. Furthermore, for every $x\in M$ and 
$\tilde\mu\in\widetilde{\mathfrak M}_1$, we have 
\[
\frac{\int_{TM} U_z(x)a(x)\, d\tilde\mu(x,v)}{\int_{TM} a(x)\, d\tilde\mu(x,v)}
=
-\frac{\int_{TM} h^\infty(x,z)a(x)\, d\tilde\mu(x,v)}{\int_{TM} a(x)\, d\tilde\mu(x,v)}
+
\hat u_0(z)
\leqslant  0,
\]
which implies that $U_z\in\mathcal S$. Then $u_0(x)\geqslant   U_z(x)$ for all $x\in M$. In particular, by taking $x=z\in \mathcal A$, we get 
\[
u_0(z)\geqslant  U_z(z)=-h^\infty(z,z)+\hat u_0(z)=\hat u_0(z).
\]
We have thus shown that $u_0\geqslant  \hat u_0$ on $\mathcal A$.
Since $\hat u_0$ is a subsolution of \eqref{E0} and $u_0$ is a solution,
Proposition \ref{prop AC} yields $u_0\geqslant  \hat u_0$ on $M$.
Therefore, $u_0=\hat u_0$.
\end{proof}

\begin{lem}\label{h=w}
Let $w$ be a subsolution of \eqref{E0}. Let $x_1,x_2\in M$. If $d_H(x_1,x_2)=0$, we have $h^\infty(x_1,x_2)=w(x_2)-w(x_1)$.
\end{lem}
\begin{proof}
By Proposition \ref{h>w} we know that for a subsolution $w$ of \eqref{E0},
\[h^\infty(x_1,x_2)\geqslant w(x_2)-w(x_1).\]
Assume that
\[h^\infty(x_1,x_2)>w(x_2)-w(x_1).\]
Adding $h^\infty(x_2,x_1)$ to both sides of the above inequality, we obtain
\[0=h^\infty(x_2,x_1)+h^\infty(x_1,x_2)>h^\infty(x_2,x_1)+w(x_2)-w(x_1),\]
which implies
\[w(x_1)-w(x_2)>h^\infty(x_2,x_1).\]
This leads to a contradiction.
\end{proof}

\begin{cor}\label{h-h=c}
By Proposition \ref{h>w}, $-h^\infty(\cdot,x)$ is a subsolution of \eqref{E0}. If $d_H(x_1,x_2)=0$, then by Lemma \ref{h=w},
\[h^\infty(x_1,x_2)=-h^\infty(x_2,x)-\big(-h^\infty(x_1,x)\big),\]
that is, $h^\infty(x_2,x)$ and $h^\infty(x_1,x)$ differ by a constant $h^\infty(x_1,x_2)$.
\end{cor}

\noindent{\it Proof of Corollary \ref{cor1}.} For each $x\in M$, the map $y\mapsto-h^\infty(y,x)$ is a subsolution of \eqref{E0}. Taking $x_0,y\in\mathcal C$, it follows from Lemma \ref{h=w} that
\[-h^\infty(x_0,x)-\big(-h^\infty(y,x)\big)=h^\infty(y,x_0).\]
Therefore, by Theorem \ref{thm1}, the maximal solution $u_\lambda$ of \eqref{E} converges uniformly to
\begin{align*}u_0(x)&=\inf_{\tilde\mu\in \widetilde{\mathfrak M}_{\mathcal C}}
\frac{\int_{TM} a(y)h^\infty(y,x)\,d\tilde\mu(y,v)}{\int_{TM} a(y)\,d\tilde\mu(y,v)}
\\&=\inf_{\tilde\mu\in \widetilde{\mathfrak M}_{\mathcal C}}
\frac{\int_{TM} a(y)\big(h^\infty(y,x_0)+h^\infty(x_0,x)\big)\,d\tilde\mu(y,v)}{\int_{TM} a(y)\,d\tilde\mu(y,v)}
\\&=C_{\mathcal C}+h^\infty(x_0,x).
\end{align*}
\qed

\section{Examples}

Examples 1 and 2 exhibit classes of problems for which hypothesis
{\rm (S)} holds for the original equation, while Example 3 shows
that this is not true in general.

\subsection{Example 1}\label{ex1}
Let $U\in C^\infty(\R)$ be a skew periodic function, i.e., there exists a constant $b\in\R$ such that $U(x)=\ol U(x)+bx$, where $\ol U\in C^\infty(\R)$ is a $1$-periodic function. Moreover, following \cite[Section 3]{GL}, we assume that $U(x)$ has finitely many critical points in $[0,1)$ indexed as follows. Assume there are $k$ (stable) local minima $x_1,x_2,\dots,x_k$ of $U(x)$, interleaved by $k$ (unstable) local maxima $x_{\frac{1}{2}},x_{1+\frac{1}{2}},\dots,x_{k+\frac{1}{2}}=x_{\frac{1}{2}}+1$. According to \cite[Lemma 4.1]{GL}, the projected Aubry set of 
\begin{equation}\label{e12}
u'(x)\bigl(u'(x)-U'(x)\bigr)=0
  \quad \textrm{in}\quad \mathbb S^1\simeq [0,1),
\end{equation}
is $\mathcal A=\{x_i,x_{i+\frac{1}{2}}\}_{i\in\{1,\dots,k\}}$. Note that in this case, Mather measures are convex combinations of Dirac measures $\delta_{(x_i,0)}$ or $\delta_{(x_{i+\frac12},0)}$.  Consider
\begin{equation}\label{e11}
  \lambda a(x)u(x)+u'(x)\bigl(u'(x)-U'(x)\bigr)=0
  \quad \textrm{in}\quad \mathbb S^1\simeq [0,1),
\end{equation}
We take $\emptyset\neq \{y_j\}_{j\in J}\subsetneq\mathcal A$ and let $a(y_j)>0$ for all $j\in J$ and $a(x)<0$ for all $x\in\mathcal A\setminus \{y_j\}_{j\in J}$. 
Then condition {\rm (A)} is satisfied. If hypothesis {\rm (S)} holds, by Theorem \ref{thm1}, the maximal solution of \eqref{e11}
uniformly converges to
\[
u_0(x)=\sup_{w\in\mathcal S}w(x),
\]
where
\[
\mathcal S=\big\{w\mid w\text{ is a subsolution of }\eqref{e12}
\text{ with }a(y_j)w(y_j)\leqslant0
\text{ for all }j\in J\big\}.
\]
Let $h^\infty(\cdot,\cdot)$ be the Peierls barrier of \eqref{e12}.
Since the Mather measures supported on
$\mathcal A_1=\{y_j\}_{j\in J}$ are convex combinations of the
Dirac Mather measures supported at the points $y_j$,
Theorem \ref{thm1} yields
\[
u_0(x)
=
\inf_{\substack{\theta_j\geqslant0\\
\sum_{j\in J}\theta_j=1}}
\frac{\displaystyle\sum_{j\in J}
\theta_j a(y_j)h^\infty(y_j,x)}
{\displaystyle\sum_{j\in J}\theta_j a(y_j)}
=
\min_{j\in J}h^\infty(y_j,x).
\]
Indeed, the quotient in the first line is a convex combination of
$h^\infty(y_j,x)$ with weights proportional to $\theta_j a(y_j)$.

We next give a large class of choices for which hypothesis {\rm (S)}
is automatically satisfied. Assume that $\emptyset\neq\mathcal A_1\subsetneq\mathcal A$
satisfies
\begin{equation}\label{eq:comp}
x_{i+\frac12}\in\mathcal A_1
\quad\Longrightarrow\quad
x_i,x_{i+1}\in\mathcal A_1,
\end{equation}
and let
\[
a>0\quad\text{on }\mathcal A_1,
\qquad
a<0\quad\text{on }\mathcal A\setminus\mathcal A_1.
\]
Choose $\varepsilon>0$ so small that
\[
4\varepsilon<
\min_i\Big\{
U(x_{i+\frac12})-U(x_i),\,
U(x_{i+\frac12})-U(x_{i+1})
\Big\}.
\]
Prescribe
\[
v(x_i)=
\begin{cases}
-2\varepsilon,&x_i\in\mathcal A_1,\\
\varepsilon,&x_i\notin\mathcal A_1,
\end{cases}
\qquad
v(x_{i+\frac12})=
\begin{cases}
-\varepsilon,&x_{i+\frac12}\in\mathcal A_1,\\
2\varepsilon,&x_{i+\frac12}\notin\mathcal A_1.
\end{cases}
\]
Condition \eqref{eq:comp} implies that, on every interval between two consecutive critical
points, the prescribed increment of $v$ has the same sign as that of
$U$ and has strictly smaller absolute value. Hence $v$ can be extended
to a periodic $C^1$ function satisfying
\[
v'=\theta_i U',
\qquad 0<\theta_i<1,
\]
on each such interval. Therefore
\[
v'(v'-U')
=-\theta_i(1-\theta_i)|U'|^2<0
\]
away from the critical points.

At every $x\in\mathcal A$ we have $a(x)v(x)<0$. Since $\mathcal A$ is
finite, there exist a neighborhood $N$ of $\mathcal A$ and a constant
$\kappa>0$ such that
\[
a(x)v(x)\leqslant-\kappa
\qquad\text{on }N.
\]
Moreover, since $v'(v'-U')<0$ on the compact set
$\mathbb S^1\setminus N$, there exists $\eta>0$ such that
\[
v'(v'-U')\leqslant-\eta
\qquad\text{on }\mathbb S^1\setminus N.
\]
Thus, for all sufficiently small $\lambda>0$,
\[
\lambda a(x)v(x)+v'(x)\bigl(v'(x)-U'(x)\bigr)
\leqslant
\begin{cases}
-\lambda\kappa,&x\in N,\\[1mm]
-\eta/2,&x\in\mathbb S^1\setminus N.
\end{cases}
\]
Hence $v$ is a strict subsolution of \eqref{e11} for every sufficiently
small $\lambda>0$. Since $v\geqslant-2\varepsilon$, hypothesis
{\rm (S)} is satisfied.

\begin{figure}[htbp]
    \centering
    \includegraphics[width=1\textwidth]{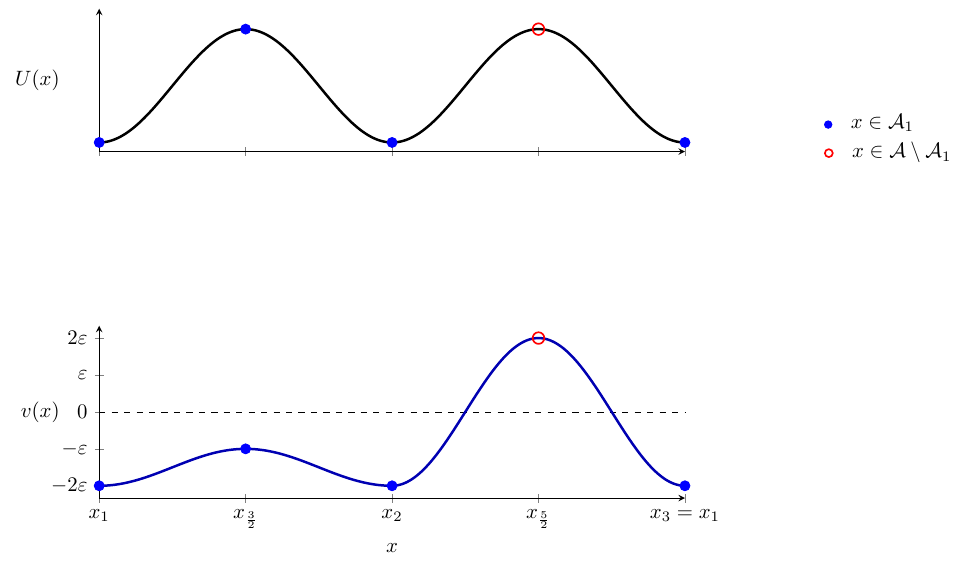}
    \caption{Construction of $v(x)$.}
    \label{fig:ex1}
\end{figure}

Consequently, for every nonempty proper subset
$\mathcal A_1\subsetneq\mathcal A$ satisfying \eqref{eq:comp}, we have $u_\lambda\to
\min\limits_{y\in\mathcal A_1}h^\infty(y,\cdot)$
uniformly as $\lambda\to0^+$. By Lemma \ref{u0<A}, applied to the Dirac Mather measures
$\delta_{(y,0)}$ with $y\in\mathcal A_1$, we have $u_0(y)\leqslant0$ for all $y\in\mathcal A_1$.
Since $0\in\mathcal S$, we also have $u_0\geqslant0$, and hence
\[
u_0=0\qquad\text{on }\mathcal A_1.
\]
Since $\mathcal A_1\subsetneq\mathcal A$ and \eqref{eq:comp} holds, there exists at least one local maximum
$x_{i+\frac12}\notin\mathcal A_1$. At such a point, $v(x_{i+\frac12})=2\varepsilon$.
Moreover, one can check that $v\in\mathcal S$. Hence,
\[
u_0(x_{i+\frac12})\geqslant v(x_{i+\frac12})
=2\varepsilon>0.
\]
Consequently, the sign-changing discount selects a nonconstant critical
solution, whereas the solution selected in \cite{Da4} is $\min\limits_{y\in\mathcal A}h^\infty(y,x)\equiv0$,
as noted in \cite{GL}.

\medskip

As a simple example, let $U:\mathbb S^1\to\mathbb R$ be a smooth function having exactly two critical points $0$ (minimum) and $X$ (maximum) and take $b=0$. Then
\[h^\infty(0,x)=U(x)-U(0),\qquad h^\infty(X,x)\equiv 0.\]
If $a(0)>0$ and $a(X)<0$, then the maximal solution of \eqref{e11} uniformly converges to
\[U(x)-U(0).\]

\subsection{Example 2}\label{ex2}
Consider
\begin{equation}\label{e21}
\lambda a(x)u(x)+H\big(x,Du(x)\big)=0\qquad\textrm{in }M,
\end{equation}
where $H(x,p)=H(x,-p)$, $\max_{x\in M}H(x,0)=0$ and $H(x,0)$ has finitely many maximizers $\{x_i\}_{i=1}^k$. By \cite[Proposition 7]{Da3}, the projected Aubry set of
\[
H\big(x,Du(x)\big)=0\qquad\textrm{in }M
\]
is $\mathcal A=\{x_i\}_{i=1}^k$. Let $\emptyset\neq\{y_j\}_{j\in J}\subsetneq\{x_i\}_{i=1}^k$, and choose $a(y_j)>0$ for $j\in J$ and $a(x_i)<0$ for $x_i\notin\{y_j\}_{j\in J}$. If hypothesis {\rm (S)} of Theorem \ref{thm1} is satisfied, then the same computation as in Section \ref{ex1} shows that the maximal solution of \eqref{e21} converges uniformly to
\[u_0(x)=\min_{j\in J}h^\infty(y_j,x).
\]
Here the critical solution selected by the limit in \cite{Da4} is $\min\limits_{1\leqslant i\leqslant k}h^\infty(x_i,x)$.

\medskip

The hypothesis {\rm (S)} can hold without adding $-A\lambda$. As a concrete example, consider
\begin{equation}\label{e212}
\lambda a(x)u(x)+\big(u'(x)\big)^2-U(x)=0
\quad\textrm{in}\quad\mathbb S^1\simeq[0,1),
\end{equation}
where $U:\mathbb S^1\to\mathbb R$ satisfies
\[
U\geqslant 0,\qquad \min_{\mathbb S^1}U=0,
\qquad U(0)=U(X)=0,\qquad U(x)>0\quad \text{for }x\neq 0,\ X.
\]
Then $\mathcal A=\mathcal M=\{0,X\}$. Let $a(0)>0$ and $a(X)<0$. By continuity, we may choose disjoint neighborhoods $N_0$ and $N_X$
of $0$ and $X$, respectively, such that
\[
a>0\quad\text{on }N_0,\quad
a<0\quad\text{on }N_X,
\]
and choose $\psi\in C^\infty(\mathbb S^1)$ such that
\[
\psi=-1\quad\text{on }N_0,\quad
\psi=1\quad\text{on }N_X,\quad \psi'=0\quad\text{on }N:=N_0\cup N_X.
\]
Then there exists $\kappa>0$ such that
\[
a(x)\psi(x)\leqslant-\kappa
\qquad\text{on }N.
\]
Since $U$ vanishes only at $0$ and $X$, there exists $\eta>0$ such that
\[
U(x)\geqslant\eta
\qquad\text{on }\mathbb S^1\setminus N.
\]
For $v=\varepsilon\psi$, choosing $\varepsilon>0$ sufficiently small gives
\[
(v')^2-U
=\varepsilon^2|D\psi|^2-U
\leqslant-\frac{\eta}{2}
\qquad\text{on }\mathbb S^1\setminus N.
\]
Hence, for all sufficiently small $\lambda>0$,
\[
(Dv)^2-U+\lambda a(x)v<0
\qquad\text{in }\mathbb S^1.
\]
Thus $v$ is a strict subsolution of \eqref{e212}. Since $v$ is fixed
and bounded from below, hypothesis {\rm (S)} holds.
Let $\widetilde U$ be the periodic extension of $U$ to $\mathbb R$ and define
\[p_\pm(x):=\pm\sqrt{\widetilde U(x)}.\]
There exists $s_1\in[0,1)$ such that
\[
\int_0^{s_1}p_+(x)\,dx+\int_{s_1}^1p_-(x)\,dx=0,
\]
and $s_2\in[X,X+1)$ such that
\[
\int_X^{s_2}p_+(x)\,dx+\int_{s_2}^{X+1}p_-(x)\,dx=0.
\]
Define
\[u_1(x)=
\begin{cases}
\displaystyle\int_0^xp_+(y)\,dy,&x\in[0,s_1],\\[1ex]
\displaystyle-\int_x^1p_-(y)\,dy,&x\in(s_1,1),
\end{cases}
\]
and
\[u_2(x)=
\begin{cases}
\displaystyle\int_X^xp_+(y)\,dy,&x\in[X,s_2],\\[1ex]
\displaystyle-\int_x^{X+1}p_-(y)\,dy,&x\in(s_2,X+1).
\end{cases}
\]
After periodic extension and projection onto $\mathbb S^1$, $u_1$ and $u_2$ are viscosity solutions of
\[
  \big(u'(x)\big)^2-U(x)=0
  \quad \textrm{in}\quad \mathbb S^1.
\]
We have
\[h^\infty(0,x)=u_1(x),\qquad h^\infty(X,x)=u_2(x).\]
Therefore, if $a(0)>0$ and $a(X)<0$, the maximal solution of \eqref{e212} converges uniformly to $u_1$; if $a(0)<0$ and $a(X)>0$, it converges uniformly to $u_2$. The critical solution selected by \cite{Da4} is $\min\{u_1,u_2\}$; see \cite[Section 5.1]{Zbook}.

\subsection{Example 3}\label{sec:4.3}
The preceding examples show that hypothesis {\rm (S)} may hold without needing a term $-A\lambda$. We now give an example showing that this is not true in general.

Let $M=\mathbb S^1\simeq[0,1)$, set
\[r(x):=\sin(2\pi x),\]
fix $\eta\in(0,1)$, and consider
\[
H(x,p):=\bigl(p-r(x)\bigr)^2-\eta^2r(x)^4.
\]
The critical value is $c_0=0$. Indeed, if $w'(x)=r(x)$, then $w$ is periodic and
\[H(x,Dw)=-\eta^2r(x)^4\leqslant0,\]
so $c_0\leqslant0$. On the other hand, $H(0,p)=p^2$, hence the associated Lagrangian satisfies $L(0,0)=0$; the closed measure $\delta_{(0,0)}$ therefore gives $c_0\geqslant0$. Moreover, the above critical subsolution is strict outside $\{0,\frac12\}$, while both $\delta_{(0,0)}$ and $\delta_{(\frac12,0)}$ are minimizing. Consequently,
\[\mathcal A=\mathcal M=\Big\{0,\frac12\Big\}.\]
Take
\[a(x):=-\cos(2\pi x),\qquad \mathcal A_1=\Big\{\frac12\Big\}.\]
Then $a(\frac12)=1$, $a(0)=-1$, and condition {\rm (A)} is satisfied.

\medskip

Consider now the shifted family from Remark \ref{rem:shift},
\begin{equation}\label{e3A}
\lambda a(x)u(x)+H\big(x,u'(x)\big)-A\lambda=0,
\qquad A\geqslant0.
\end{equation}
The zero level of $H$ has the two branches
\[
p_\pm(x)=r(x)\pm\eta r(x)^2.
\]
Both branches are positive on $(0,\frac12)$ and negative on $(\frac12,1)$. Set
\[
\varepsilon_*:=\int_0^{1/2}p_-(x)\,dx
=\frac1\pi-\frac\eta4>0.
\]
Every viscosity solution $u$ of the critical equation $H(x,Du)=0$ is Lipschitz and satisfies the equation almost everywhere. Hence,
\[
u\Big(\frac12\Big)-u(0)\geqslant\varepsilon_*.
\]
Suppose that there exist $\lambda_n\to0^+$ and solutions $u_{\lambda_n}$ of \eqref{e3A} which are uniformly bounded. The same argument as in Lemma \ref{leqA}, applied to the two Dirac Mather measures, gives
\[u_{\lambda_n}\Big(\frac12\Big)\leqslant A,
\qquad u_{\lambda_n}(0)\geqslant-A.
\]
By coercivity, the family $\{u_{\lambda_n}\}$ is equi-Lipschitz; after extraction, $u_{\lambda_n}\to u$ uniformly, where $H(x,Du)=0$. Therefore
\[
\varepsilon_*
\leqslant u\Big(\frac12\Big)-u(0)
\leqslant2A.
\]
Consequently, if
\[
0\leqslant A<A_*:=\frac{\varepsilon_*}{2}
=\frac{1}{2\pi}-\frac{\eta}{8},
\]
there is no uniformly bounded sequence of solutions of \eqref{e3A} as $\lambda\to0^+$. In particular, the analogue of {\rm (S)}
cannot hold for such $A$. Otherwise, Lemmas \ref{n11} and \ref{bM} would yield a
uniformly bounded family of maximal solutions. Moreover, hypothesis {\rm (S)} fails for the original equation
corresponding to $A=0$.

\section*{Acknowledgements}

Panrui Ni is supported by the JSPS grant: KAKENHI \# 26KF0103 and by the National Natural Science Foundation of China (Grant No. 12571197). Jun Yan is supported by the National Natural Science Foundation of China (Grant Nos. 12171096, 12231010, 12571197). 
Maxime Zavidovique is supported by ANR CoSyDy (ANR-CE40-0014).
M.Z.
gratefully acknowledges support from the Institute for Advanced Study, of which he was a 
member when this work was finished.

\section*{Declarations}

\noindent {\bf Conflict of interest statement:} The authors state that there is no conflict of interest.

\medskip

\noindent {\bf Data availability statement:} Data sharing is not applicable to this article as no datasets were generated or analysed during the current study.

\end{document}